\documentclass[UTF8,11pt]{article}
\usepackage{amssymb,amsmath,color,amsfonts,float,amscd,amsthm,bm,mathrsfs,array,tikz,dutchcal}
\usepackage[title]{appendix}
\usepackage[colorlinks,urlcolor=blue,citecolor=blue,linkcolor=blue]{hyperref}
\usepackage{ulem}

\newcommand*{\circled}[1]{\lower.7ex\hbox{\tikz\draw (0pt, 0pt)%
    circle (.5em) node {\makebox[1em][c]{\small #1}};}}
 \allowdisplaybreaks

\catcode`!=11
\let\!int\int \def\int{\displaystyle\!int}
\let\!lim\lim \def\lim{\displaystyle\!lim}
\let\!sum\sum \def\sum{\displaystyle\!sum}
\let\!sup\sup \def\sup{\displaystyle\!sup}
\let\!inf\inf \def\inf{\displaystyle\!inf}
\let\!cap\cap \def\cap{\displaystyle\!cap}
\let\!max\max \def\max{\displaystyle\!max}
\let\!min\min \def\min{\displaystyle\!min}

\catcode`!=12

\let\oldsection\section
\renewcommand\section{\setcounter{equation}{0}\oldsection}

\allowdisplaybreaks

\newtheorem{defn}{Definition}[section]
\newtheorem{thm}{Theorem}[section]
\newtheorem{prop}[thm]{Proposition}
\newtheorem{lem}[thm]{Lemma}
\newtheorem{cor}[thm]{Corollary}
\newtheorem{rem}[thm]{Remark}

\newcommand{\RNum}[1]{\uppercase\expandafter{\romannumeral #1\relax}}
\newcommand{\be}{\begin{equation}}
\newcommand{\ee}{\end{equation}}
\newcommand{\bes}{\begin{eqnarray*}}
\newcommand{\ees}{\end{eqnarray*}}

\begin{document}

\title{Sharp interface limit for inhomogeneous incompressible Navier-Stokes/Allen-Cahn system in a bounded domain via a relative energy method}
\author{Song Jiang\footnote{Institute of Applied Physics and Computational Mathematics. Email: jiang@iapcm.ac.cn}
\and
Xiangxiang Su\footnote{School of Mathematical Sciences, Shanghai Jiao Tong University, Shanghai 200240, P. R. China. Email: sjtusxx@sjtu.edu.cn}
\and
Feng Xie
\footnote{Corresponding author.
School of Mathematical Sciences, and CMA-Shanghai, Shanghai Jiao Tong University, Shanghai 200240, P. R. China.
Email: tzxief@sjtu.edu.cn}
}
\date{}
\maketitle

\noindent{\bf Abstract:}
This paper concerns the sharp interface limit of solutions to the inhomogeneous incompressible Navier-Stokes/Allen-Cahn coupled system in a bounded domain $\Omega \subset  \mathbb{R}^n,\ n =2,3$. Based on a relative energy method, we prove that the solutions to the Navier-Stokes/Allen-Cahn system converge to the corresponding solutions to a sharp interface model provided that the thickness of the diffuse interfacial zone goes to zero. It is noted that the relative energy method can avoid both the spectral estimates of the linearized Allen-Cahn operator and the construction of approximate solutions by the matched asymptotic expansion method in the study of the sharp interface limit process. And some suitable functionals are designed and estimated by elaborated energy methods accordingly.

\vskip0.2cm


\section{Introduction and Main Results}
\hspace{2em}
The widely accepted diffuse interface model of a two-phase flow is called as ``Model H", which describes the motion of two macroscopically immiscible incompressible viscous Newtonian fluids with matched density. It was usually governed by the Navier-Stokes/Cahn-Hilliard coupled system and was introduced in \cite{Gurtin, Hohenberg}. One can refer to \cite{Elliott C. M.} and the references cited therein for more physical background. For model H, \cite{Abels H0} discussed the  convergence of solutions to the diffuse interface model with different densities to the sharp interface model, while \cite{Abels H1} prescribed an evolution law for the interface that takes the diffusion effects into account. In Model H, the order parameter is in accordance with the law of conservation of mass. Otherwise, it becomes the Allen-Cahn equation. One can refer to \cite{Chen X0} for the derivation of these two models. And the sharp interface limit for the two-phase flow of incompressible viscous fluids is an interesting but challenging problem and has been widely studied. In \cite{Liu Y N} the authors constructed an approximate solution by the matched asymptotic expansion method and proved that the solution to the Stokes/Allen-Cahn system converges to the related solution to a sharp interface model, while \cite{Fei} studied the similar limit problem for the Navier-Stokes/Allen-Cahn system. Also refer to \cite{JSX} by the same authors of this paper.  Recently, the relative entropy method, which calculates the quantitative stability of the appropriate distance between two solutions, was popularly applied to study the sharp interface limit problems.
For example, for the Allen-Cahn equation of the mean curvature flow, the convergence rate of sharp interface limit was given in \cite{Fischer J.}. Since such a method does not depend on the spectral analysis of the linearized Allen-Cahn operator and on the comparison principle, so that it could be applied to other much more complex models, such as the Navier-Stokes/Allen-Cahn coupled system. For instance, \cite{Sebastian Hensel} studied the sharp interface limit in both 2D and 3D bounded domains. We also refer to \cite{Fischer J1} for the related study, while the planar multi-phase mean curvature flow  was
analyzed in \cite{Fischer J2}.
The main purpose of this paper is to establish these results for the inhomogeneous incompressible Navier-Stokes/Allen-Cahn coupled system in a general parameter regime where the variation of density is also involved. Moreover, the curvature is not included in the kinetic boundary condition for the related free boundary value problem. We can justify the sharp interface limit for this case by the relative entropy method.

We consider the phase field model based on the second-order Allen-Cahn operator rather than the H model. This model also has been studied in \cite{Chen M, Li Y} etc.  More specifically, we are concerned with the inhomogeneous incompressible Navier-Stokes/Allen-Cahn coupled system in a smooth bounded domain $\Omega \subseteq \mathbb{R}^n,\ n=2, 3$:
\begin{subequations} \label{9.15.1}
\begin{align}
\rho_{\varepsilon} \partial_t\mathbf{v}_\varepsilon + \rho_{\varepsilon} \mathbf{v}_\varepsilon \cdot \nabla \mathbf{v}_\varepsilon - \Delta \mathbf{v}_{\varepsilon}+\nabla p_{\varepsilon} &=-\varepsilon \operatorname{div} (\nabla c_{\varepsilon} \otimes \nabla c_{\varepsilon}-\frac{1}{2}|\nabla c_{\varepsilon}|^2I) && \text { in } \Omega \times(0, T_{1}), \label{9.6.4}\\
\partial_t\rho_{\varepsilon}+\operatorname{div} (\rho_{\varepsilon}\mathbf{v}_{\varepsilon}) &=0 && \text { in } \Omega \times (0, T_{1}), \label{2.17.1}\\
\operatorname{div} \mathbf{v}_{\varepsilon} &=0 && \text { in } \Omega \times (0, T_{1}), \\
\rho_{\varepsilon} \partial_{t} c_{\varepsilon}+\rho_{\varepsilon}  \mathbf{v}_{\varepsilon} \cdot \nabla c_{\varepsilon} &=-m_{\varepsilon} \mu_{\varepsilon} && \text { in } \Omega \times (0, T_{1}), \label{10.19.1}\\
\rho_{\varepsilon} \mu_{\varepsilon} &=- {\varepsilon} \Delta c_{\varepsilon}+\frac{\rho_{\varepsilon} }{\varepsilon} f^{\prime}\left(c_{\varepsilon}\right) & & \text { in } \Omega \times\left(0, T_{1}\right), \label{9.11.1}\\
\rho_{\varepsilon}(\cdot, 0) =\rho_{\varepsilon, 0} > 0, \quad \mathbf{v}_{\varepsilon}(\cdot, 0) &=\mathbf{v}_{\varepsilon, 0}, \quad c_{\varepsilon}(\cdot, 0) =c_{\varepsilon, 0} \in[-1,1] & & \text { in } \Omega, \\
\left(\mathbf{v}_{\varepsilon},c_{\varepsilon}, \mu_{\varepsilon}\right) &=(0,-1,0) & & \text { on } \partial \Omega \times\left(0, T_{1}\right) \label{9.28.1},
\end{align}
\end{subequations}
where $\mathbf{v}_{\varepsilon}$, $ p_{\varepsilon}$, $\rho_{\varepsilon}$ are velocity vector, pressure and density of the fluid respectively. $c_{\varepsilon}$ denotes the order parameter and $\mu_{\varepsilon}$ stands for chemical potential, $m_{\varepsilon}\triangleq m_0 \,\varepsilon^\theta$ expresses a mobility coefficient.
And $\varepsilon$ is a small positive parameter which represents the ``width" of the interfacial region and $f$ is a double well potential. In general, $f$  satisfies
\begin{eqnarray*}
f \in C^\infty(\mathbb{R}),\quad f'(\pm1)=0,\quad f''(\pm1)> 0,\quad f(c)=f(-c)> 0, \quad \forall c \in(-1,1).
\end{eqnarray*}
A typical example $f(c)=\frac{1}{8}(1-c^2)^2$. However, our analysis will not depend on the explicit expression of $f$.

Notice that \cite{Abels H} gave a (non-)convergence result under the condition that $\theta>1$. In this paper we will focus on the related convergence result provided that $\theta\in (-\frac{1}{4}, 1]$. Without loss of generality, we only deal with the case that $m_0=1$ and $\theta=0$ for simplicity in the following discussion. It is direct to check that our results can be extended to the case that $\theta\in (-\frac{1}{4}, 1]$ without increasing any difficulty in analysis which will be explained in details in Remark \ref{RE}.

As is known that the solution $c_{\varepsilon}$ to the Allen-Cahn equation converges to a function $2\chi-1$ as $\varepsilon$ tends to 0. And an important fact about $c_{\varepsilon}$ is that
$$
-1 \le c_{\varepsilon}(x,t) \le 1 \quad \text{ for any }  \varepsilon \ge 0, \, x \in \Omega, \, t \in (0,T_1).
$$

Below, we will show that the sharp interface limit of (\ref{9.15.1}) is the following free boundary value problem as the parameter $\varepsilon$ tends to zero.
\begin{subequations} \label{3.3.1}
\begin{align}
\rho \partial_t\mathbf{v}+\rho \mathbf{v}\cdot \nabla \mathbf{v}-\Delta \mathbf{v}+\nabla p & =0 && \text { in } \Omega^{\pm}(t),\  t \in (0, T_{0}), \label{11.8.1}\\
\partial_t\rho+\operatorname{div} (\rho \mathbf{v}) &=0  && \text { in } \Omega(t),\  t \in (0, T_{0}), \label{2.20.11}\\
\operatorname{div} \mathbf{v} & =0 && \text { in } \Omega(t),\  t \in (0, T_{0}), \\
\partial_t\chi+ \mathbf{v}\cdot \nabla \chi& =0 && \text { in } \Omega(t),\  t \in (0, T_{0}), \\
{[2 D \mathbf{v}-p \mathbf{I}] \mathbf{n}_{\Gamma_{t}}} & =c_0 H_{\Gamma_{t}} \mathbf{n}_{\Gamma_{t}} && \text { on } \Gamma_{t},\  t \in (0, T_{0}), \label{9.17.1}\\
{[\mathbf{v}]} & =0 && \text { on } \Gamma_{t},\  t \in (0, T_{0}), \\
V_{\Gamma_{t}}-\mathbf{n}_{\Gamma_{t}} \cdot \mathbf{v}|_{\Gamma_{t}} & =0 && \text { on } \Gamma_{t},\  t \in (0, T_{0}) \label{11.8.2},\\
\rho(\cdot, 0) =\rho_{0}>0, \quad \mathbf{v}(\cdot, 0) &=\mathbf{v}_{0}, \quad \chi(\cdot, 0) =\chi_{0} & & \text { in } \Omega, \\
(\mathbf{v},\chi) &=(0,0) & & \text { on } \partial \Omega \times (0, T_{0}) \label{9.28.1}.
\end{align}
\end{subequations}
Here $\Omega$ is closed by two smooth domains $\Omega^\pm(t)$ which are separated by a free boundary $\Gamma_t$  for each $t\in(0, T_0)$.  $D \mathbf{v}=\frac{1}{2}(\nabla\mathbf{v}+(\nabla\mathbf{v})^{\top})$ is the stress tensor. Moreover, $V_{\Gamma_{t}}$ and $H_{\Gamma_t}$ are the normal velocity and (mean) curvature of the interface $\Gamma_t$, and $\mathbf{n}_{\Gamma_{t}}$ is the outward normal vector to $\Omega^-(t)$. The definition of $c_0$ is given by $c_0=\int_{-1}^{1} \sqrt{2 f(r)} \,\mathrm{d} r $.

In (\ref{9.17.1}), $[h]$ means the jump of $h$ across $\Gamma_t$ whose definition is given in the following:
$$
[h](p, t)=\lim _{d \rightarrow 0+}[h\left(p+\mathbf{n}_{\Gamma_{t}}(p) d\right)-h\left(p-\mathbf{n}_{\Gamma_{t}}(p) d\right)].
$$
Finally, the transport equations (\ref{2.17.1}) and (\ref{2.20.11}) ensure that
\begin{align}
0 < \|\rho_{\varepsilon}(t)\|_{L^\infty}=\|\rho_{\varepsilon,0}\|_{L^\infty} < \infty, \quad 0 < \|\rho(t)\|_{L^\infty}=\|\rho_0\|_{L^\infty} < \infty.
\end{align}

Now, we introduce the relative energy for the models (\ref{9.15.1}) and (\ref{3.3.1}) as follows:
\begin{align} \label{8.23.6}
E \left[\rho_{\varepsilon},\mathbf{v}_{\varepsilon}, c_{\varepsilon} \mid \rho, \mathbf{v}, \chi\right]
\triangleq \frac{1}{2} \int_{\Omega}(\rho_{\varepsilon}-\rho)^2 + \rho_{\varepsilon}\left|\mathbf{v}_{\varepsilon}-\mathbf{v} \right|^{2}\,\mathrm{d} x+ \int_{\Omega} \frac{\varepsilon}{2}\left|\nabla c_{\varepsilon}\right|^{2}+\frac{\rho_{\varepsilon}}{\varepsilon} f\left(c_{\varepsilon}\right)-\boldsymbol{\xi} \cdot \nabla \psi_{\varepsilon}\,\mathrm{d} x,
\end{align}
and we also define the measure for the difference in the phase indicators as
\begin{align} \label{9.13.1}
E_{\mathrm{vol}}\left[c_{\varepsilon} \mid \chi\right] \triangleq \int_{\Omega}|c_0 \chi-\psi_{\varepsilon}||\vartheta(d_\Gamma)| \,\mathrm{d} x,
\end{align}
where
\begin{align} \label{9.12.1}
\psi_{\varepsilon}=\int_{-1}^{c_{\varepsilon}} \sqrt{2 \rho_\varepsilon(r) f(r)} \,\mathrm{d} r. \end{align}
The vector field $\boldsymbol{\xi}$ in (\ref{8.23.6}) is a suitable extension of the unit normal vector field  $\mathbf{n}_{\Gamma_{t}}$, and $\vartheta$ in (\ref{9.13.1}) will be a smooth truncation of the signed distance function. The definitions and more properties of $\boldsymbol{\xi}$  and $\vartheta$ are introduced in details in Section 2.

\begin{thm} \label{thm1.1}
Assume that the system of equations (\ref{9.15.1}) admits a solution $(\rho_{\varepsilon},\mathbf{v}_{\varepsilon},c_{\varepsilon},\mu_{\varepsilon})$ on a time interval $[0, T_1]$ with $T_{1} \in(0, \infty)$, and $(\rho,\mathbf{v},\chi)$ is a strong solution to the sharp interface limit model (\ref{3.3.1}) in $[0, T_0]\ (T_0\leq T_1)$ in the sense of Definition \ref{def2.1}. Moreover, the initial data satisfy the assumption
\begin{align} \label{9.15.2}
E \left[\rho_{\varepsilon},\mathbf{v}_{\varepsilon}, c_{\varepsilon} \mid \rho, \mathbf{v}, \chi\right](0)+E_{\mathrm{vol}}\left[c_{\varepsilon} \mid \chi\right](0) \le C_0 \varepsilon^\frac{1}{3}, \quad \forall \varepsilon \in(0,1)
\end{align}
for some $C_0>0$, where $E \left[\rho_{\varepsilon},\mathbf{v}_{\varepsilon}, c_{\varepsilon} \mid \rho, \mathbf{v}, \chi\right]$ and $E_{\mathrm{vol}}\left[c_{\varepsilon} \mid \chi\right]$ are defined in(\ref{8.23.6}) and (\ref{9.13.1}) respectively. Then there exists a time $T \in (0,T_0)$ and constants $C=C(\mathbf{v},\chi,T)>0,\ \varepsilon_0\in(0,1]$, such that the following estimate
\begin{align} \label{2.21.7}
E \left[\rho_{\varepsilon},\mathbf{v}_{\varepsilon}, c_{\varepsilon} \mid \rho, \mathbf{v}, \chi\right](t) + E_{\mathrm{vol}}\left[c_{\varepsilon} \mid \chi\right](t) \le C \varepsilon^\frac{1}{3}
\end{align}
holds true for any $\varepsilon \in(0,\varepsilon_0)$ and almost every $t\in(0,T)$.

Furthermore,
\begin{align*}
&\frac{1}{3} \| \nabla (\mathbf{v}_{\varepsilon}-\mathbf{v})\|_{L^2((0,t)\times \Omega)}^2+\frac{1}{3} \| \mu_\varepsilon\|_{L^2((0,t)\times \Omega)}^2\\
+&\int_0^t \int_{\Omega} \frac{1}{2}(\operatorname{div} \boldsymbol{\xi} \,\sqrt{\frac{2 f}{\rho_\varepsilon }} + \mu_\varepsilon )^2 \, \mathrm{d} x \, \mathrm{d} t +\int_0^t \int_{\Omega} \varepsilon (\frac{\mu_\varepsilon }{2} +\frac{\vartheta}{\rho_{\varepsilon}} |\nabla c_{\varepsilon}|)^2 \,\mathrm{d} x \, \mathrm{d} t \le C \varepsilon^\frac{1}{3}.
\end{align*}
\end{thm}

\noindent
As a direct consequence of Theorem \ref{thm1.1}, we obtain:
\begin{cor} \label{cor1.2}
 Assume that the initial data satisfy  (\ref{9.15.2}). Then there exist universal constants ${R}$ and $\varepsilon_0 > 0$, where ${R}$ is independent of $\varepsilon$, such that for every $T \in (0,T_0)$ and $\varepsilon \in(0,\varepsilon_0)$,
\begin{align}
\underset{0 \le t \le T}{sup} \|\psi_{\varepsilon}(x, t)-c_0\chi(x, t)\|_{L^1(\Omega)} \le R\,\varepsilon^\frac{1}{6}.
\end{align}
\end{cor}

Before proceeding, it is necessary to recall the motivation of this paper and related results. First, when the densities of the two fluids are the same, especially $\rho=1$, then the system of equations (\ref{9.15.1}) is reduced into the incompressible Navier-Stokes /Allen-Cahn coupled system, which has been considered in \cite{Fei,Liu Y N,Sebastian Hensel,JSX} and references cited therein. Consequently, the main results and methods in this paper are also valid for the sharp interface limit problem of incompressible Navier-Stokes/Allen-Cahn coupled system. However, the variation of density indeed increases additional difficulties in analysis which will be explained later. Second, as mentioned above that when $m_{\varepsilon}=m_0{\varepsilon}^{\theta}$ the (non-)convergence result was proved in \cite{Abels H} under the condition that $\theta>1$. However, when $-\frac{1}{4} < \theta \le 1$ in (\ref{10.19.1}), we can prove the sharp interface limit between the solutions to (\ref{9.15.1}) and the solutions to (\ref{3.3.1}) in this work. Therefore, based on the results in \cite{Abels H} and this paper, it gives a classification about the sharp interface limit with respect to the parameter $\theta$ in the regime of $\theta\in (-\frac{1}{4}, \infty)$.

Next, let us explain the main difficulties in analysis and the strategy of proof.
The first main difficulty comes from the capillary term of $\operatorname{div} (\nabla c_{\varepsilon} \otimes \nabla c_{\varepsilon}-\frac{1}{2}|\nabla c_{\varepsilon}|^2I)$. It is in fact a singular and unbounded term due to appearing of a strong layer of $c_{\varepsilon}$ across the interfacial region. We adopt the strategy of combining it with $ \nabla \psi_{\varepsilon} \cdot \partial_{t} \boldsymbol{\xi}$ together. Then the corresponding singular terms can be canceled with each other. However, the time evolution of the vector field $\boldsymbol{\xi}$ should be estimated in a suitable way. One refers to Section 3  for details. In addition, as mentioned above, the variation in density also causes some analytical difficulties due to the fact that the densities of the two fluids often differ significantly.

Finally, it is noted that the convergence results for the incompressible case were recently established in \cite{Sebastian Hensel} by the relative energy method. However, for some technical reasons, the evolution equation of the free boundary (\ref{11.8.2}) is replaced by $V_{\Gamma_{t}}-\mathbf{n}_{\Gamma_{t}} \cdot \mathbf{v}|_{\Gamma_{t}} =H_{\Gamma_{t}} $ in \cite{Sebastian Hensel}, also see \cite{Fei, Liu Y N}. In general, it seems that the velocity of the free boundary should only be related to the velocity of the fluid,
and it is less related to the curvature of the curve describing the free boundary. Consequently, this is the main motivation why the curvature is not included in (\ref{11.8.2}) in this paper. Thus, to close the estimates and avoid the appearance of $H_{\Gamma_{t}}$, it is necessary to deal with the term of $(\operatorname{div} \boldsymbol{\xi} \,\sqrt{\frac{2 f}{\rho_\varepsilon }})^2$ in (\ref{3.5.2}) carefully. It is also emphasized that the case $\theta\in[-1, -1/4]$ is interesting and will be left for future study.

The paper is organized as follows. Some symbols and facts which will be used frequently are stated in Section 2. Section 3 is devoted to deriving the corresponding estimates of relative energy and phase indicators. Finally, in Section 4, we complete the proofs of Theorem \ref{thm1.1} and Corollary \ref{cor1.2} based on the estimates achieved in Section 3.

\section{Preliminaries}
We start with the definition of strong solutions to the sharp interface limit model (\ref{3.3.1}).
\begin{defn} \label{def2.1}
$(\rho,\mathbf{v},\chi)$ is called a strong solution to the free boundary problem for the system of equations (\ref{3.3.1}),  if for all $T\in(0,T_0)$ the triple pair $(\rho,\mathbf{v},\chi)$ satisfies the following  requirements:\\
i) It holds
$$
\begin{cases}
(\rho,\mathbf{v}) \in H^1(0,T;L^2(\Omega)) \cap L^2(0,T;H^1(\Omega)),\\
\mathbf{v} \in W^{1, \infty}\left([0, T] ; W^{1, \infty}(\Omega)\right) \cap C_t^1 C_x^0(\bar{\Omega} \times[0, T] \backslash \Gamma_t) \cap C_t^0 C_x^2(\bar{\Omega} \times[0, T] \backslash \Gamma_t).
\end{cases}
$$
Furthermore, there exists a constant C depending on T, such that
\begin{align*}
(\rho,\mathbf{v}) \in L^\infty(0,T;C^1(\Omega))
\end{align*}
and the derivatives of $(\rho,\mathbf{v})$ satisfy
\begin{align*}
\| \nabla \rho \|_{L^\infty(\Omega \times (0,T))}\le C, \quad \| \nabla \mathbf{v} \|_{L^\infty(\Omega \times (0,T))} \le C.
\end{align*}
ii) The velocity field $\mathbf{v}$ satisfies both $\nabla \cdot \mathbf{v}(\cdot, t)=0$ and the equation of momentum balance in the sense of distribution. That is,
\begin{align*}
& \int_{\Omega} \rho(\chi(\cdot, T)) \mathbf{v}(\cdot, T) \cdot \eta(\cdot, T)  \,\mathrm{d} x-\int_{\Omega} \rho(\chi_0) \mathbf{v}_0 \cdot \eta(\cdot, 0)  \,\mathrm{d} x \\
& =\int_0^T \int_{\Omega} \rho(\chi) \mathbf{v} \cdot \partial_t \eta  \,\mathrm{d} x  \,\mathrm{d} t +\int_0^T \int_{\Omega} \rho(\chi) \mathbf{v} \otimes \mathbf{v} : \nabla \eta \,\mathrm{d} x  \,\mathrm{d} t \\
& \quad-\int_0^T \int_{\Omega} \nabla \mathbf{v}: \nabla \eta  \,\mathrm{d} x  \,\mathrm{d} t+c_0 \int_0^T \int_{\Gamma_t} H_{\Gamma_{t}} \mathbf{n}_{\Gamma_{t}} \cdot \eta \,\mathrm{d} \mathcal{H}^{d-1}\,\mathrm{d} t
\end{align*}
holds true for almost every $T \in(0, T_0)$ and all $\eta \in C_c^{\infty}(\Omega \times[0, T])$ satisfying $\nabla \cdot \eta=0$.

For this definition, one can also refer to \cite[Definition 10]{Hensel 1} and \cite[Definition 4]{Sebastian Hensel} for more details.
\end{defn}

Let us define the signed distance function
$$
d_{\Gamma}(x, t)\triangleq \operatorname{sdist}(\Gamma_{t}, x)=
\begin{cases}
\operatorname{dist}\left(\Omega^{-}(t), x\right) & \text { if } x \notin \Omega^{-}(t), \\
-\operatorname{dist}\left(\Omega^{+}(t), x\right) & \text { if } x \in \Omega^{-}(t).\end{cases}
$$

Choose a suitably small positive constant $\delta$ such that $\operatorname{dist}(\partial \Omega, \Gamma_{t})> 3 \delta$. For $t\in (0, T_0)$, we introduce the tubular neighborhood of $\Gamma_t$ as follows:
\begin{align*}
\Gamma_t(\delta)\triangleq \{y\in \Omega: \hbox{dist}(y, \Gamma_t)<\delta\},\quad \Gamma(\delta)=\bigcup_{t\in (0, T_0)}\Gamma_t(\delta)\times\{t\}.
\end{align*}

For every point $x \in \Gamma_{t}$, there exists a local diffeomorphisms $X_{0}: \mathbb{T}^{1} \times (0, T_{0}) \rightarrow \Gamma_{t}$. For simplicity of notation, set
$$
\mathbf{n}_{\Gamma_{t}}(x, t)\triangleq \mathbf{n}(s, t), \quad \text { for all } x=X_{0}(s, t) \in \Gamma_{t}.
$$
Then
$$
d_{\Gamma}\left(X_{0}(s, t)+r \mathbf{n}(s, t),\ t\right)=r,
$$
which implies that for all $(x,t)\in \Gamma(3 \delta)$ and $(p,t)\in \Gamma$,
\begin{eqnarray} \label{9.19.1}
\nabla d_{\Gamma}(x, t)=\mathbf{n}_{\Gamma_{t}}\left(P_{\Gamma_{t}}(x), t\right),\quad \partial_{t} d_{\Gamma}(x, t)=-V_{\Gamma_{t}}\left(P_{\Gamma_{t}}(x), t\right),\quad \Delta d_{\Gamma}(p, t)=-H_{\Gamma_{t}}(p, t),
\end{eqnarray}
where $P_{\Gamma_{t}}(x)$ is the orthogonal projection (cf. \cite[Section 4.1]{Chen X}).

By the way, we introduce the extended  mean curvature vector $\mathbf{H}(x,t)$ of $\Gamma_{t}$ to the whole domain $\Omega$.
\begin{defn}
The extended mean curvature vector $\mathbf{H}(x,t)$ is defined by
\begin{align} \label{2.16.1}
\mathbf{H}(x,t) \triangleq  H_{\Gamma_{t}}(P_{\Gamma_{t}}(x),t)\mathbf{n}_{\Gamma_{t}}(x,t)\zeta(x,t),
\end{align}
where $x=P_{\Gamma_{t}}(x)+d_{\Gamma}(x,t) \mathbf{n}_{\Gamma_{t}}(x,t)$ and $\zeta$ is a cut-off function satisfying
\begin{align*}
\zeta(\cdot,t) \in C^\infty_c(\Gamma_t(2\delta))\quad  \text{ and } \quad  \zeta(\cdot,t)=1 \quad \text{ in } \Gamma_t(\delta).
\end{align*}
\end{defn}

We also define the constant extension of $\mathbf{v}$ away from $\Gamma_{t}$ by
\begin{align} \label{2.23.2}
\mathbf{\tilde{v}}(x,t) \triangleq \mathbf{v}(P_{\Gamma_{t}}(x),t)  \quad \text{ in }\ \Gamma_{t}(2\delta).
\end{align}
Using the regularity of $\mathbf{v}$, we have
\begin{align} \label{2.23.3}
|\mathbf{\tilde{v}}(x,t) -\mathbf{v}(x,t)| \le \omega(t) |d_{\Gamma}(x, t)|, \quad |\nabla \mathbf{\tilde{v}}(x,t) -\nabla \mathbf{v}(x,t)| \le \tilde{\omega}(t) |d_{\Gamma}(x, t)|
\end{align}
for some non-negative bounded functions $\omega(t)$ and $\tilde{\omega}(t)$.

The main results in this paper rely on the choice of the following two functions $\boldsymbol{\xi}$ and $\vartheta$. Firstly, following the general strategy of \cite{Sebastian Hensel}, we extend inner normal vector as
\begin{align} \label{2.27.1}
\boldsymbol{\xi}(x, t)=\phi\left(\frac{\mathrm{d}_{\Gamma}(x, t)}{\delta}\right) \nabla \mathrm{d}_{\Gamma}(x, t),
\end{align}
where $\phi(x) \geq 0$ is an even, smooth function on $\mathbb{R}$, which satisfies monotonically decreasing on $x \in[0,1]$ and
$$
\begin{cases}\phi(x)>0\quad \text { for } & |x|<1, \\ \phi(x)=0\quad \text { for } & |x| \geq 1, \\ 1-4 x^2 \leq \phi(x) \leq 1-\frac{1}{2} x^2 & \text { for }|x| \leq 1 / 2.\end{cases}
$$
These ensure that $\phi'(x)\sim O(x)$ in the interval $[-\frac{1}{2}, \frac{1}{2}]$. Based on these definitions, some properties of $\boldsymbol{\xi}$ are summarized as follows. For every $T\in(0,T_0)$,\\
(A1) Regularity estimates
\begin{align}\label{9.17.3}
&\boldsymbol{\xi} \in C^{0,1}\left([0, T] ; L^{\infty}(\Omega)\right) \cap L^{\infty} \left([0, T] ; C_c^{1,1}(\Omega)\right).
\end{align}
Besides, it satisfies that
\begin{align} \label{2.28.1}
\left\|(\partial_t \boldsymbol{\xi}, \nabla^2 \boldsymbol{\xi})\right\|_{L^{\infty}(\Omega \times(0, T))} \leq C .
\end{align}
(A2) Coercivity and consistency:
\begin{align}
|\boldsymbol{\xi}| & \leq 1-c \min \left\{d_{\Gamma}^2, 1\right\} & & \text { a.e. on } \Omega \times[0, T], \notag\\
\boldsymbol{\xi} & = \mathbf{n}_{\Gamma_{t}} \text { and } \nabla \cdot \boldsymbol{\xi}=- H_{\Gamma_{t}} & & \text { on } \Gamma_{t} .\label{2.18.2}
\end{align}
(A3) It also holds
\begin{align}
&|\mathbf{H} \cdot \boldsymbol{\xi}+\nabla \cdot \boldsymbol{\xi}|  \leq C \min \{d_{\Gamma}, 1\} && \text { a.e. on } \Omega \times[0, T], \label{2.20.1}\\
&(\boldsymbol{\xi} \cdot \nabla)\mathbf{H}=0, \quad (\boldsymbol{\xi} \cdot \nabla)\mathbf{\tilde{v}}=0 \quad &&\text{ in } \Omega \label{2.23.1}.
\end{align}
(A4) Moreover, there exists a constant $C$ such that
\begin{align}
&\left|\partial_t \boldsymbol{\xi}+(\mathbf{v} \cdot \nabla) \boldsymbol{\xi}+(\nabla \mathbf{v})^{\top} \boldsymbol{\xi} \right| \leq C \min \{d_{\Gamma}, 1\} & \text { a.e. on } \Omega \times[0, T], \label{2.16.5}\\
&\left|\boldsymbol{\xi} \cdot\left(\partial_t+\mathbf{v} \cdot \nabla\right) \boldsymbol{\xi}\right| \leq C \min \left\{d_{\Gamma}^2, 1\right\} & \text { a.e. on } \Omega \times[0, T]  \label{2.16.7}.
\end{align}

The proofs of (A1) and (A2) are the direct consequences of the definition (\ref{2.27.1}). Next, we only give a brief proof of (A3) and (A4).

For (A3), we obtain from (\ref{9.19.1}) that
\begin{align*}
\operatorname{div} \boldsymbol{\xi} & =\left|\nabla \mathrm{d}_{\Gamma}\right|^2 \phi'\left(\frac{\mathrm{d}_{\Gamma}(x, t)}{\delta}\right)+\phi\left(\frac{\mathrm{d}_{\Gamma}(x, t)}{\delta}\right) \Delta d_{\Gamma} =O (d_{\Gamma})-\mathbf{H} \cdot \boldsymbol{\xi}.
\end{align*}
According to (\ref{2.16.1}), we know that $\mathbf{H}(x,t)$  is a constant extension of $H_{\Gamma_{t}}(P_{\Gamma_{t}}(x),t)$
in $\Gamma_t(\delta)$. Moreover, $\boldsymbol{\xi} \equiv 0$ outside $\Gamma_t(\delta)$ and we conclude that $(\boldsymbol{\xi} \cdot \nabla)\mathbf{H}=0$ in $\Omega $.

Similarly, $(\boldsymbol{\xi} \cdot \nabla)\mathbf{\tilde{v}}=0 $ holds straightly by (\ref{2.23.2}).

It suffices to verify (A4) in the region $\Gamma_t(\delta)$ because $\boldsymbol{\xi}$ and its derivatives vanish outside this region.

Recalling (\ref{11.8.2}) and (\ref{9.19.1}), we deduce that
\begin{align} \label{2.16.2}
\partial_t \mathrm{d}_{\Gamma}+\mathbf{\tilde{v}} \cdot \nabla \mathrm{d}_{\Gamma}=0 \quad \text { in } \Gamma_t(\delta) .
\end{align}
As a consequence,
\begin{align} \label{2.16.3}
\partial_t \nabla \mathrm{d}_{\Gamma}+(\mathbf{\tilde{v}} \cdot \nabla) \nabla \mathrm{d}_{\Gamma}+(\nabla \mathbf{\tilde{v}})^{\top} \nabla \mathrm{d}_{\Gamma}=0 \quad \text { in } {\Gamma}_t(\delta)
\end{align}
and
\begin{align} \label{2.16.4}
\partial_t \phi\left(\frac{\mathrm{d}_{\Gamma}(x, t)}{\delta}\right)+\mathbf{\tilde{v}} \cdot \nabla \phi\left(\frac{\mathrm{d}_{\Gamma}(x, t)}{\delta}\right)=0 .
\end{align}
Combining (\ref{2.23.3}), (\ref{2.16.3}) and (\ref{2.16.4}) together yields (\ref{2.16.5}).

Employing (\ref{2.27.1}) and testing (\ref{2.16.5}) by $\boldsymbol{\xi}$ arrive at (\ref{2.16.7}).

Below, we give some comments on the rationality of the definition of relative energy. For this purpose, it is convenient to rewrite
\begin{align}
\mathbf{n}_\varepsilon=\frac{\nabla c_\varepsilon}{|\nabla c_\varepsilon|}=\frac{\nabla \psi_\varepsilon}{|\nabla \psi_\varepsilon|}.
\end{align}
Then it is natural to rewrite (\ref{8.23.6}) as the following form:
\begin{align*}
&E \left[\rho_{\varepsilon},\mathbf{v}_{\varepsilon}, c_{\varepsilon} \mid \rho, \mathbf{v}, \chi\right]\\
=& \frac{1}{2}\int_{\Omega} (\rho_{\varepsilon}-\rho)^2+\rho_{\varepsilon}\left|\mathbf{v}_{\varepsilon}-\mathbf{v} \right|^{2}\,\mathrm{d} x+\int_{\Omega} \frac{\varepsilon}{2}\left|\nabla c_{\varepsilon}\right|^2+\frac{\rho_{\varepsilon}}{\varepsilon} f(c_{\varepsilon})-\left|\nabla \psi_{\varepsilon}\right| \mathrm{d} x+\int_{\Omega} (1-\boldsymbol{\xi} \cdot \mathbf{n}_{\varepsilon})\left|\nabla \psi_{\varepsilon}\right| \mathrm{d} x\\
=& \frac{1}{2}\int_{\Omega} (\rho_{\varepsilon}-\rho)^2+ \rho_{\varepsilon}\left|\mathbf{v}_{\varepsilon}-\mathbf{v} \right|^{2}\,\mathrm{d} x+\int_{\Omega} \frac{1}{2}\left(\sqrt{\varepsilon}| \nabla c_{\varepsilon}|-\frac{\sqrt{2 \rho_{\varepsilon} f (c_{\varepsilon})}}{\sqrt{\varepsilon}}\right)^2 \mathrm{~d} x+\int_{\Omega} (1-\boldsymbol{\xi} \cdot \mathbf{n}_{\varepsilon})\left|\nabla \psi_{\varepsilon}\right| \mathrm{d} x\\
\ge& \frac{1}{2}\int_{\Omega} (\rho_{\varepsilon}-\rho)^2+ \rho_{\varepsilon}\left|\mathbf{v}_{\varepsilon}-\mathbf{v} \right|^{2}\,\mathrm{d} x+\int_{\Omega} \frac{1}{2}\left(\sqrt{\varepsilon}| \nabla c_{\varepsilon}|-\frac{\sqrt{2 \rho_{\varepsilon} f (c_{\varepsilon})}}{\sqrt{\varepsilon}}\right)^2 \mathrm{~d} x+\int_{\Omega}\frac{1}{2}|\mathbf{n}_{\varepsilon}-\boldsymbol{\xi}|^2\left|\nabla \psi_{\varepsilon}\right| \mathrm{d} x,
\end{align*}
where the interfacial contribution $(1-\boldsymbol{\xi} \cdot \mathbf{n}_{\varepsilon})\left|\nabla \psi_{\varepsilon}\right|$ controls the interface error.

Furthermore, we introduce some coercivity properties of the relative energy which will be used frequently in this paper.
\begin{prop} \label{prop2.1}
For every $T \in [0, T_0)$, there exists a constant $C=C(\chi, v, T)$, such that for all $t \in [0, T]$, it holds that
\begin{gather}
 \int_{\Omega} (\rho_{\varepsilon}-\rho)^2+ \rho_{\varepsilon}\left|\mathbf{v}_{\varepsilon}-\mathbf{v} \right|^{2}\,\mathrm{d} x+\int_{\Omega}\left(\sqrt{\varepsilon}\left|\nabla c_{\varepsilon}\right|-\frac{\sqrt{2 \rho_{\varepsilon} f(c_{\varepsilon})}}{\sqrt{\varepsilon}} \right)^2 \mathrm{~d} x \leq C E \left[\rho_{\varepsilon},\mathbf{v}_{\varepsilon}, c_{\varepsilon} \mid \rho, \mathbf{v}, \chi\right], \\
\int_{\Omega} \min \{d_{\Gamma}^2, 1 \}\left(\frac{\varepsilon}{2}\left|\nabla c_{\varepsilon}\right|^2+\frac{\rho_{\varepsilon}}{\varepsilon} f(c_{\varepsilon})\right) \mathrm{d} x \leq C E \left[\rho_{\varepsilon},\mathbf{v}_{\varepsilon}, c_{\varepsilon} \mid \rho, \mathbf{v}, \chi\right],\\
\int_{\Omega} (1-\mathbf{n}_{\varepsilon} \cdot \boldsymbol{\xi})\left|\nabla \psi_{\varepsilon}\right|+|\mathbf{n}_{\varepsilon}-\boldsymbol{\xi}|^2\left|\nabla \psi_{\varepsilon}\right| \mathrm{d} x +\int_{\Omega} \min \{d_{\Gamma}^2, 1\}\left|\nabla \psi_{\varepsilon}\right|  \mathrm{d} x  \leq C E \left[\rho_{\varepsilon},\mathbf{v}_{\varepsilon}, c_{\varepsilon} \mid \rho, \mathbf{v}, \chi\right], \\
\int_{\Omega} |\mathbf{n}_{\varepsilon}-\boldsymbol{\xi}|^2 \varepsilon\left|\nabla c_{\varepsilon}\right|^2  \mathrm{d} x +\int_{\Omega} \min \left\{d_{\Gamma}^2, 1\right\} \varepsilon\left|\nabla c_{\varepsilon}\right|^2  \mathrm{d} x  \leq C E \left[\rho_{\varepsilon},\mathbf{v}_{\varepsilon}, c_{\varepsilon} \mid \rho, \mathbf{v}, \chi\right]
\end{gather}
and
\begin{align}
\int_{\Omega}\left(\min \{d_{\Gamma}, 1\}+\sqrt{1-\mathbf{n}_{\varepsilon} \cdot \boldsymbol{\xi}}\right) \left| \varepsilon | \nabla \varphi_{\varepsilon}|^2-|\nabla \psi_{\varepsilon}| \right| \mathrm{d} x \leq C  E \left[\rho_{\varepsilon},\mathbf{v}_{\varepsilon}, c_{\varepsilon} \mid \rho, \mathbf{v}, \chi\right].
\end{align}
\end{prop}
\begin{proof}
Since the proof of Proposition \ref{prop2.1} is similar to \cite[Lemma 5]{Sebastian Hensel}, we therefore omit it for simplicity.
\end{proof}

Based on Proposition \ref{prop2.1}, we have the following lemma.
\begin{lem} \label{lemma 2.3}
There exists a generic constant $C<\infty$, which is independent of $\varepsilon$, such that for any $t \in[0, T]$ the following estimate holds.
$$
\int_0^t \!\int_{\Omega} \nabla \mathbf{v}: (I-\mathbf{n}_{\varepsilon} \otimes \mathbf{n}_{\varepsilon}) (\varepsilon|\nabla c_{\varepsilon}|^2- |\nabla \psi_{\varepsilon}|) \,\mathrm{d} x \, \mathrm{d}\varsigma \leq C \int_0^t \! E \left[\rho_{\varepsilon},\mathbf{v}_{\varepsilon}, c_{\varepsilon} \mid \rho, \mathbf{v}, \chi\right] \,\mathrm{d}\varsigma.
$$
\end{lem}
\begin{proof}
Notice that $\nabla \mathbf{v}: I=\operatorname{div} \mathbf{v} =0$, it suffices to prove that
$$
\int_0^t \!\int_{\Omega} \nabla \mathbf{v}: \mathbf{n}_{\varepsilon} \otimes \mathbf{n}_{\varepsilon} (\varepsilon|\nabla c_{\varepsilon}|^2- |\nabla \psi_{\varepsilon}|) \,\mathrm{d} x \, \mathrm{d}\varsigma \leq C \int_0^t  E \left[\rho_{\varepsilon},\mathbf{v}_{\varepsilon}, c_{\varepsilon} \mid \rho, \mathbf{v}, \chi\right]\,\mathrm{d}\varsigma.
$$
We adopt the facts that $\mathbf{n}_{\varepsilon}=\mathbf{n}_{\varepsilon}-\boldsymbol{\xi}+\boldsymbol{\xi}$ and $\nabla \mathbf{v}: \mathbf{n}_{\varepsilon} \otimes \boldsymbol{\xi}=(\boldsymbol{\xi} \cdot \nabla) \mathbf{v} \cdot \mathbf{n}_{\varepsilon}$ to yield that
\begin{align*}
&\int_0^t \!\int_{\Omega} \nabla \mathbf{v}: \mathbf{n}_{\varepsilon} \otimes \mathbf{n}_{\varepsilon} (\varepsilon|\nabla c_{\varepsilon}|^2- |\nabla \psi_{\varepsilon}|) \,\mathrm{d} x \, \mathrm{d}\varsigma\\
= & \int_0^t \!\int_{\Omega} \nabla\mathbf{v}:\left(\mathbf{n}_{\varepsilon} \otimes (\mathbf{n}_{\varepsilon}-\boldsymbol{\xi})\right) (\varepsilon|\nabla c_{\varepsilon}|^2- |\nabla \psi_{\varepsilon}|) \,\mathrm{d} x \, \mathrm{d}\varsigma \\
&+\int_0^t \!\int_{\Omega} (\boldsymbol{\xi} \cdot \nabla) (\mathbf{v}-\mathbf{\tilde{v}}) \cdot \mathbf{n}_{\varepsilon} (\varepsilon|\nabla c_{\varepsilon}|^2- |\nabla \psi_{\varepsilon}|) \,\mathrm{d} x \, \mathrm{d}\varsigma+\int_0^t \!\int_{\Omega} (\boldsymbol{\xi} \cdot \nabla) \mathbf{\tilde{v}} \cdot \mathbf{n}_{\varepsilon} (\varepsilon|\nabla c_{\varepsilon}|^2- |\nabla \psi_{\varepsilon}|) \,\mathrm{d} x \, \mathrm{d}\varsigma\\
\le & C\int_0^t \!\int_{\Omega} |\mathbf{n}_{\varepsilon}-\boldsymbol{\xi}| \sqrt{\varepsilon} |\nabla c_{\varepsilon}| \left|\sqrt{\varepsilon}| \nabla c_{\varepsilon}|-\varepsilon^{-\frac{1}{2}} \sqrt{2 \rho_{\varepsilon} f(c_{\varepsilon})}\right| \,\mathrm{d} x \, \mathrm{d}\varsigma \\
&+C\int_0^t \!\int_{\Omega} \min \{d_\Gamma, 1\}\left(\varepsilon|\nabla c_{\varepsilon}|^2 -|\nabla \psi_{\varepsilon}|\right) \,\mathrm{d} x \, \mathrm{d}\varsigma,
\end{align*}
where (\ref{2.23.3}) and (\ref{2.23.1}) are used in the last inequality.

Furthermore, applying the Cauchy-Schwarz inequality and Proposition \ref{prop2.1}, we finish the proof of Lemma \ref{lemma 2.3}.
\end{proof}

Consider a mapping
$$
\vartheta: \bar{\Omega} \times (0, T_{0}) \mapsto [-1,1],
$$
which stands for a time-dependent weight function with the following properties (cf. \cite[Section 3.3]{Sebastian Hensel}).

For any $T \in(0, T_0)$, we assume that there exist two positive constants $c \in(0,1)$ and $C\in(1, \infty)$, which may depend on $\chi$ and $T$ such that:\\
(B1) Regularity
\begin{align}
&\vartheta \in C^{0,1}\left([0, T] ; L^{\infty}(\bar{\Omega})\right) \cap L^{\infty} \left([0, T] ; C^{0,1}(\bar{\Omega})\right). \label{9.17.3}
\end{align}
Also, we require a further assumption on the regularity estimates of the derivatives
\begin{align} \label{9.17.4}
\left\|\left(\partial_t \vartheta, \nabla \vartheta\right)\right\|_{L^{\infty}(\Omega \times(0, T))} \leq C .
\end{align}
(B2) Coercivity and consistency:
\begin{align} \label{9.17.5}
c \min \{d_{\Gamma}, 1\} \leq |\vartheta| \leq C \min \{d_{\Gamma}, 1\} \quad \text { is fulfilled in } \Omega \times[0, T].
\end{align}
Additionally, for all $t \in (0, T)$, we suppose
\begin{align*}
\vartheta(\cdot, t)<0 & \text { a.e. in the interior of }\{\chi(\cdot, t)=1\} \cap \Omega
\end{align*}
and
\begin{align*}
\vartheta(\cdot, t)>0 & \text { a.e. in the interior of }\{\chi(\cdot, t)=0\} \cap \Omega,
\end{align*}
respectively.\\
(B3) Transportability property
\begin{align} \label{9.17.6}
\left|\partial_t \vartheta+ (\mathbf{v} \cdot \nabla) \vartheta\right| \leq C \min \{d_{\Gamma}, 1\} \quad \text { a.e. in } \Omega \times[0, T].
\end{align}

In particular, $\vartheta$ can be constructed as the smooth odd truncation of signed distance function which takes the following form:
$$
\vartheta(r)=\left\{
\begin{array}{rll}
-\delta & \text { as } & r \geq  \delta, \\
-r & \text { as } & -\delta/2\leq r \leq  \delta/2, \\
\delta & \text { as } & r \leq - \delta.
\end{array}\right.
$$
Then (B1) and (B2) are satisfied immediately. And (B3) is a straightforward consequence of (\ref{2.23.3}), (\ref{2.16.2}) and chain rule.

In view of (\ref{9.17.5}), we have
\begin{align} \label{9.23.4}
\int_\Omega \min \{d_{\Gamma}, 1\} |c_0 \chi-\psi_{\varepsilon}| \,\mathrm{d} x \leq CE_{\mathrm{vol}}\left[c_{\varepsilon} \mid \chi\right] .
\end{align}

\begin{lem} \label{lemma2.2}
For a suitably small $\eta>0$, it holds true
\begin{equation} \label{2.19.7}
\int_{\Omega} |c_0 \chi-\psi_{\varepsilon}| |\mathbf{v}_{\varepsilon}-\mathbf{v}| \,\mathrm{d} x \le \frac{C}{\eta}(E \left[\rho_{\varepsilon},\mathbf{v}_{\varepsilon}, c_{\varepsilon} \mid \rho, \mathbf{v}, \chi\right]+E_{\mathrm{vol}}\left[c_{\varepsilon} \mid \chi\right])+\eta \int_{\Omega}  |\nabla \mathbf{v}_{\varepsilon}-\nabla \mathbf{v}|^2 \,\mathrm{d} x.
\end{equation}
\end{lem}
\begin{proof}
The proof is referred to (31) in \cite{Sebastian Hensel}.
\end{proof}

\section{Gronwall-type Estimates}
In this section, we move to derive the differential inequality about the relative energy (\ref{8.23.6}) and the difference in the phase indicators (\ref{9.13.1}).
\subsection{Estimate for the Relative Energy}
The purpose of this subsection is to derive the estimate of the relative energy.
By the definition of $E \left[\rho_{\varepsilon},\mathbf{v}_{\varepsilon}, c_{\varepsilon} \mid \rho, \mathbf{v}, \chi\right]$, it is necessary to establish the estimates of error functions. Denote
\begin{align} \label{2.21.1}
\mathbf{w}=\mathbf{v}_\varepsilon -\mathbf{v}.
\end{align}

\begin{prop} \label{prop3.2}
Let $\mathbf{w}$ be defined as above and $E \left[\rho_{\varepsilon},\mathbf{v}_{\varepsilon}, c_{\varepsilon} \mid \rho, \mathbf{v}, \chi\right]$ be defined as in (\ref{8.23.6}). Then there exists a positive constant $C$, such that for any $T \in (0,T_0)$, the following inequality holds:
\begin{eqnarray}
\begin{split} \label{9.12.3}
& E \left[\rho_{\varepsilon},\mathbf{v}_{\varepsilon}, c_{\varepsilon} \mid \rho, \mathbf{v}, \chi\right](T)+\int_0^T \!\int_{\Omega} \frac{1}{2}(\nabla \mathbf{w})^2 + \frac{1}{2}(\mu_\varepsilon )^{2}\,\mathrm{d} x \, \mathrm{d} t+\int_0^T \!\int_{\Omega} \frac{1}{2}(\operatorname{div} \boldsymbol{\xi} \,\sqrt{\frac{2 f}{ \rho_\varepsilon }} + \mu_\varepsilon)^2 \,\mathrm{d} x \, \mathrm{d} t\\
\le& E \left[\rho_{\varepsilon},\mathbf{v}_{\varepsilon}, c_{\varepsilon} \mid \rho, \mathbf{v}, \chi\right](0)+ C \varepsilon^\frac{1}{3}+ C \int_0^T (E \left[\rho_{\varepsilon},\mathbf{v}_{\varepsilon}, c_{\varepsilon} \mid \rho, \mathbf{v}, \chi\right](t)+E_{\mathrm{vol}}\left[c_{\varepsilon} \mid \chi\right](t))\, \mathrm{d} t.
\end{split}
\end{eqnarray}
\end{prop}

\begin{proof}
$\bullet$ Consider the difference between equations (\ref{2.17.1}) and (\ref{2.20.11}), it follows that
\begin{align} \label{2.20.12}
\partial_t(\rho_{\varepsilon}-\rho)+\mathbf{v}\nabla(\rho_{\varepsilon}-\rho)=-\operatorname{div} (\rho_{\varepsilon} \mathbf{w}).
\end{align}
Multiplying (\ref{2.20.12}) by $(\rho_{\varepsilon}-\rho)$ and integrating the resulting equality over $\Omega\times[0, T]$ imply that
\begin{eqnarray}
\begin{split} \label{2.23.4}
&\int_{\Omega}\frac{1}{2} [\rho_{\varepsilon}(T)-\rho(T)]^2\,\mathrm{d} x -\int_{\Omega}\frac{1}{2} [\rho_{\varepsilon}(0)-\rho(0)]^2\,\mathrm{d} x\\ =&\int_0^T \!\int_{\Omega} \operatorname{div} (\rho_{\varepsilon}\mathbf{w})(\rho-\rho_{\varepsilon}) \,\mathrm{d} x \, \mathrm{d} t =\int_0^T \!\int_{\Omega} \nabla \rho_{\varepsilon} \mathbf{w} (\rho-\rho_{\varepsilon}) \,\mathrm{d} x \, \mathrm{d} t \\
=&\int_0^T \!\int_{\Omega} \nabla (\rho_{\varepsilon}-\rho) \,\mathbf{w} (\rho-\rho_{\varepsilon}) \,\mathrm{d} x \, \mathrm{d} t+\int_0^T \!\int_{\Omega} \nabla \rho \,\mathbf{w} (\rho-\rho_{\varepsilon}) \,\mathrm{d} x \, \mathrm{d} t\\
=&\int_0^T \!\int_{\Omega} \nabla \rho \,\mathbf{w} (\rho-\rho_{\varepsilon}) \,\mathrm{d} x \, \mathrm{d} t,
\end{split}
\end{eqnarray}
where the divergence-free conditions $\hbox{div}\ \mathbf{v}=0$ and $\hbox{div}\ \mathbf{w}=0$ are used above.\\
$\bullet$ Next, we multiply (\ref{2.17.1}) by $\frac{1}{2}\mathbf{v}^2$. And it follows from the integration by parts that
\begin{align} \label{8.23.4}
\int_{\Omega} (\frac{1}{2} \rho_\varepsilon \mathbf{v}^2)(T)\,\mathrm{d} x -\int_{\Omega} (\frac{1}{2} \rho_\varepsilon \mathbf{v}^2)(0)\,\mathrm{d} x =\int_0^T \!\int_{\Omega} \,\rho_\varepsilon \mathbf{v} \cdot \partial_{t}\mathbf{v}+\rho_\varepsilon \mathbf{v}_\varepsilon \otimes \mathbf{v}: \nabla \mathbf{v} \,\mathrm{d} x \, \mathrm{d} t.
\end{align}
Using (\ref{2.17.1}), one has
\begin{align*}
&\int_0^T \!\int_{\Omega} (\rho_{\varepsilon} \partial_t\mathbf{v}_\varepsilon + \rho_{\varepsilon} \mathbf{v}_\varepsilon \cdot \nabla \mathbf{v}_\varepsilon ) \mathbf{v}_\varepsilon \,\mathrm{d} x \, \mathrm{d} t \\
=&\int_{\Omega} \frac{1}{2} (\rho_\varepsilon \mathbf{v}_\varepsilon ^2)(T)\,\mathrm{d} x -\int_{\Omega} \frac{1}{2} (\rho_\varepsilon \mathbf{v}_\varepsilon ^2)(0)\,\mathrm{d} x +\int_0^T \!\int_{\Omega} - \frac{1}{2} (\partial_{t} \rho_\varepsilon) \mathbf{v}_\varepsilon ^2+\rho_{\varepsilon} \mathbf{v}_\varepsilon \otimes  \mathbf{v}_\varepsilon :\nabla \mathbf{v}_\varepsilon \,\mathrm{d} x \, \mathrm{d} t\\
=&\int_{\Omega} \frac{1}{2} (\rho_\varepsilon \mathbf{v}_\varepsilon ^2)(T)\,\mathrm{d} x -\int_{\Omega} \frac{1}{2} (\rho_\varepsilon \mathbf{v}_\varepsilon ^2)(0)\,\mathrm{d} x +\int_0^T \!\int_{\Omega}  \frac{1}{2} \operatorname{div} (\rho_{\varepsilon}\mathbf{v}_{\varepsilon}) \mathbf{v}_\varepsilon ^2+\rho_{\varepsilon} \mathbf{v}_\varepsilon \otimes  \mathbf{v}_\varepsilon :\nabla \mathbf{v}_\varepsilon \,\mathrm{d} x \, \mathrm{d} t\\
=&\int_{\Omega} \frac{1}{2} (\rho_\varepsilon \mathbf{v}_\varepsilon ^2)(T)\,\mathrm{d} x -\int_{\Omega} \frac{1}{2} (\rho_\varepsilon \mathbf{v}_\varepsilon ^2)(0)\,\mathrm{d} x +\int_0^T \!\int_{\Omega} -\rho_{\varepsilon} \mathbf{v}_\varepsilon \otimes  \mathbf{v}_\varepsilon :\nabla \mathbf{v}_\varepsilon +\rho_{\varepsilon} \mathbf{v}_\varepsilon \otimes  \mathbf{v}_\varepsilon :\nabla \mathbf{v}_\varepsilon \,\mathrm{d} x \, \mathrm{d} t \\
=&\int_{\Omega} \frac{1}{2} (\rho_\varepsilon \mathbf{v}_\varepsilon ^2)(T)\,\mathrm{d} x -\int_{\Omega} \frac{1}{2} (\rho_\varepsilon \mathbf{v}_\varepsilon ^2)(0)\,\mathrm{d} x ,
\end{align*}
then multiplying (\ref{9.6.4}) by $ \mathbf{v}_\varepsilon $ and integration by parts lead  to
\begin{eqnarray}
\begin{split}  \label{9.11.4}
&\int_{\Omega} \frac{1}{2} (\rho_\varepsilon \mathbf{v}_\varepsilon ^2)(T)\,\mathrm{d} x -\int_{\Omega} \frac{1}{2} (\rho_\varepsilon \mathbf{v}_\varepsilon ^2)(0)\,\mathrm{d} x +\int_0^T \!\int_{\Omega} (\nabla \mathbf{v}_{\varepsilon})^2 \,\mathrm{d} x \, \mathrm{d} t\\
=&-\int_0^T \!\int_{\Omega} \varepsilon \nabla \mathbf{v}_{\varepsilon}:(I-\mathbf{n}_{\varepsilon} \otimes \mathbf{n}_{\varepsilon}) |\nabla c_{\varepsilon}|^2 \,\mathrm{d} x \, \mathrm{d} t.
\end{split}
\end{eqnarray}
Similarly, multiplying (\ref{9.6.4}) by $ \mathbf{v}$, integrating it over $\Omega$, and using the integration by parts, we obtain
\begin{eqnarray}
\begin{split} \label{9.11.3}
&\int_{\Omega} (\rho_\varepsilon \mathbf{v}_\varepsilon \mathbf{v})(T)\,\mathrm{d} x -\int_{\Omega} (\rho_\varepsilon \mathbf{v}_\varepsilon \mathbf{v})(0)\,\mathrm{d} x \\
=&\int_0^T \!\int_{\Omega} \rho_\varepsilon \mathbf{v}_\varepsilon \partial_{t}\mathbf{v} +\rho_\varepsilon \mathbf{v}_\varepsilon \otimes \mathbf{v}_\varepsilon: \nabla \mathbf{v} -\nabla \mathbf{v}_{\varepsilon}:\nabla \mathbf{v}-\varepsilon \nabla \mathbf{v}:(I-\mathbf{n}_{\varepsilon} \otimes \mathbf{n}_{\varepsilon})|\nabla c_{\varepsilon}|^2 \,\mathrm{d} x \, \mathrm{d} t.
\end{split}
\end{eqnarray}
Consequently, from (\ref{8.23.4})-(\ref{9.11.3}), we arrive at
\begin{eqnarray}
\begin{split} \label{2.18.1}
&\int_{\Omega} \frac{1}{2} (\rho_\varepsilon \mathbf{w}^2)(T)\,\mathrm{d} x-\int_{\Omega} \frac{1}{2} (\rho_\varepsilon \mathbf{w}^2)(0)\,\mathrm{d} x + \int_0^T \!\int_{\Omega} \nabla \mathbf{v}_{\varepsilon} :\nabla \mathbf{w} \,\mathrm{d} x \, \mathrm{d} t\\
=&\int_0^T \!\int_{\Omega} -\mathbf{w} (\rho_\varepsilon \partial_{t}\mathbf{v} +\rho_\varepsilon \mathbf{v}_\varepsilon \nabla \mathbf{v} )-\varepsilon \nabla \mathbf{w} :(I-\mathbf{n}_{\varepsilon} \otimes  \mathbf{n}_{\varepsilon})|\nabla c_{\varepsilon}|^2 \,\mathrm{d} x \, \mathrm{d} t.
\end{split}
\end{eqnarray}
To deal with the second term on the left-hand side in (\ref{2.18.1}), we adopt integration by parts, (\ref{9.17.1}) and (\ref{2.18.2}) to yield that
\begin{align*}
&\int_0^T \!\int_{\Omega} \nabla \mathbf{v}_{\varepsilon} :\nabla \mathbf{w} \,\mathrm{d} x \, \mathrm{d} t=\int_0^T \!\int_{\Omega}  |\nabla \mathbf{w}|^2 +\nabla \mathbf{v} :\nabla \mathbf{w} \,\mathrm{d} x \, \mathrm{d} t\\
=&\int_0^T \!\int_{\Omega}  |\nabla \mathbf{w}|^2 -\mathbf{w}(\rho \partial_t\mathbf{v}+\rho \mathbf{v} \nabla \mathbf{v}+\nabla p )\,\mathrm{d} x \, \mathrm{d} t+\int_0^T \!\int_{\Gamma_t}c_0 H_{\Gamma_{t}} \mathbf{n}_{\Gamma_{t}} \cdot \mathbf{w} \,\mathrm{d} \mathcal{H}^{d-1} \, \mathrm{d} t\\
=&\int_0^T \!\int_{\Omega}  |\nabla \mathbf{w}|^2 -\mathbf{w}(\rho \partial_t\mathbf{v}+\rho \mathbf{v} \nabla \mathbf{v})\,\mathrm{d} x \, \mathrm{d} t-\int_0^T \!\int_{\Gamma_t}c_0 (\nabla \cdot\boldsymbol{\xi} )\mathbf{n}_{\Gamma_{t}} \cdot \mathbf{w} \,\mathrm{d} \mathcal{H}^{d-1} \, \mathrm{d} t\\
=&\int_0^T \!\int_{\Omega} |\nabla \mathbf{w}|^2 -\mathbf{w}(\rho \partial_t\mathbf{v}+\rho \mathbf{v} \nabla \mathbf{v})\,\mathrm{d} x \, \mathrm{d} t+\int_0^T \!\int_{\Omega} c_0 \chi(\mathbf{w}\cdot \nabla)(\nabla \cdot\boldsymbol{\xi}) \,\mathrm{d} x \, \mathrm{d} t.
\end{align*}
Moreover, it is direct to verify that
$$
\rho_\varepsilon \mathbf{v}_\varepsilon \nabla \mathbf{v}-\rho \mathbf{v} \nabla \mathbf{v}=[(\rho_\varepsilon-\rho)\mathbf{v}+ \rho_\varepsilon \mathbf{w}]\nabla \mathbf{v}.
$$
Hence, we conclude that
\begin{small}
\begin{eqnarray}
\begin{split} \label{2.18.4}
&\int_{\Omega} \frac{1}{2} (\rho_\varepsilon \mathbf{w}^2)(T)\,\mathrm{d} x-\int_{\Omega} \frac{1}{2} (\rho_\varepsilon \mathbf{w}^2)(0)\,\mathrm{d} x + \int_0^T \!\int_{\Omega} |\nabla \mathbf{w}|^2 \,\mathrm{d} x \, \mathrm{d} t\\
=&\int_0^T \!\int_{\Omega}  -\mathbf{w}[(\rho_\varepsilon-\rho)(\partial_{t} \mathbf{v}+\mathbf{v}\nabla\mathbf{v})+\rho_\varepsilon \mathbf{w} \nabla \mathbf{v}]-\varepsilon \nabla \mathbf{w}:(I-\mathbf{n}_{\varepsilon} \otimes \mathbf{n}_{\varepsilon})|\nabla c_{\varepsilon}|^2 - c_0 \chi(\mathbf{w}\cdot \nabla)(\nabla \cdot\boldsymbol{\xi}) \,\mathrm{d} x \, \mathrm{d} t.
\end{split}
\end{eqnarray}
\end{small}
$\bullet$ Taking  $L^2$ inner product (\ref{10.19.1}) with $\mu_\varepsilon$, after letting $m_\varepsilon=1$ and integration by parts, we get that
\begin{eqnarray}
\begin{split} \label{2.18.5}
0=&\int_0^T \!\int_{\Omega}  (\mu_\varepsilon )^{2} \,\mathrm{d} x \, \mathrm{d} t+ \int_0^T \!\int_{\Omega}  (-\varepsilon \Delta c_{\varepsilon}+\frac{\rho_\varepsilon}{\varepsilon} f^{\prime}(c_{\varepsilon})) (\partial_{t} c_{\varepsilon}+\mathbf{v}_{\varepsilon} \cdot \nabla c_{\varepsilon} )\,\mathrm{d} x \, \mathrm{d} t\\
=&\int_0^T \!\int_{\Omega}  (\mu_\varepsilon )^{2} \,\mathrm{d} x \, \mathrm{d} t+\int_0^T \!\int_{\Omega}  \varepsilon \nabla c_{\varepsilon}\partial_{t}\nabla c_{\varepsilon}-{\varepsilon} \mathbf{v}_{\varepsilon} \cdot \nabla c_{\varepsilon}  \Delta c_{\varepsilon}+\frac{\rho_\varepsilon}{\varepsilon} \partial_{t}f(c_{\varepsilon})+\frac{\rho_\varepsilon}{\varepsilon}\nabla f(c_{\varepsilon}) \cdot \mathbf{v}_{\varepsilon}\,\mathrm{d} x \, \mathrm{d} t\\
=&\int_0^T \!\int_{\Omega}  (\mu_\varepsilon )^{2} \,\mathrm{d} x \, \mathrm{d} t+\int_0^T \!\int_{\Omega}  \partial_{t}\left[\frac{\varepsilon}{2}|\nabla c_{\varepsilon}|^2+\frac{\rho_\varepsilon}{\varepsilon} f(c_{\varepsilon})\right]-{\varepsilon} \mathbf{v}_{\varepsilon} \cdot \nabla c_{\varepsilon}  \Delta c_{\varepsilon} \,\mathrm{d} x \, \mathrm{d} t\\
=&\int_0^T \!\int_{\Omega}  (\mu_\varepsilon )^{2} \,\mathrm{d} x \, \mathrm{d} t+\int_{\Omega} \left[\frac{\varepsilon}{2}|\nabla c_{\varepsilon}|^2+\frac{\rho_\varepsilon}{\varepsilon} f(c_{\varepsilon})\right](T) \,\mathrm{d} x -\int_{\Omega} \left[\frac{\varepsilon}{2}|\nabla c_{\varepsilon}|^2+\frac{\rho_\varepsilon}{\varepsilon} f(c_{\varepsilon})\right](0) \,\mathrm{d} x \\
&-\int_0^T \!\int_{\Omega} {\varepsilon} \nabla \mathbf{v}_{\varepsilon} : (I-\mathbf{n}_{\varepsilon} \otimes \mathbf{n}_{\varepsilon})|\nabla c_{\varepsilon}|^2 \,\mathrm{d} x \, \mathrm{d} t,
\end{split}
\end{eqnarray}
where the transport equation (\ref{2.17.1}) is used in the third equality. Moreover, the facts that $\operatorname{div}(\nabla c_{\varepsilon} \otimes \nabla c_{\varepsilon}-\frac{1}{2}|\nabla c_{\varepsilon}|^2I)= \nabla c_{\varepsilon}  \Delta c_{\varepsilon}$ and $\operatorname{div}\mathbf{v}_{\varepsilon} =0$ are used in the last equality.

Collecting (\ref{2.23.4}), (\ref{2.18.4}) and (\ref{2.18.5}) together leads to
\begin{small}
\begin{eqnarray}
\begin{split} \label{2.19.1}
&\int_{\Omega} \frac{1}{2} [\rho_{\varepsilon}(T)-\rho(T)]^2+\frac{1}{2}(\rho_\varepsilon \mathbf{w}^2)(T)+ \left[\frac{\varepsilon}{2}|\nabla c_{\varepsilon}|^2+\frac{\rho_\varepsilon}{\varepsilon} f(c_{\varepsilon})\right](T)\,\mathrm{d} x+\int_0^T \!\int_{\Omega} |\nabla \mathbf{w}|^2 +(\mu_\varepsilon )^{2}\, \mathrm{d} x \, \mathrm{d} t\\
=&\int_{\Omega} \frac{1}{2} [\rho_{\varepsilon}(0)-\rho(0)]^2+\frac{1}{2}(\rho_\varepsilon \mathbf{w}^2)(0)+ \left[\frac{\varepsilon}{2}|\nabla c_{\varepsilon}|^2+\frac{\rho_\varepsilon}{\varepsilon} f(c_{\varepsilon})\right](0)\,\mathrm{d} x\\
&+\int_0^T \!\int_{\Omega} -\mathbf{w}[(\rho_\varepsilon-\rho)(\partial_{t} \mathbf{v}+\mathbf{v}\nabla\mathbf{v}+\nabla \rho )+\rho_\varepsilon \mathbf{w} \nabla \mathbf{v}]+\varepsilon \nabla \mathbf{v}:(I-\mathbf{n}_{\varepsilon} \otimes \mathbf{n}_{\varepsilon})|\nabla c_{\varepsilon}|^2 - c_0 \chi(\mathbf{w}\cdot \nabla)(\nabla \cdot\boldsymbol{\xi}) \,\mathrm{d} x \, \mathrm{d} t.
\end{split}
\end{eqnarray}
\end{small}
Next, we turn to derive the estimates of $\boldsymbol{\xi} \cdot \nabla \psi_{\varepsilon}$ in the definition of relative energy (\ref{8.23.6}).\\
$\bullet$ On one hand, we employ the chain rule and (\ref{10.19.1}) to get that
\begin{eqnarray}
\begin{split} \label{2.19.2}
&\int_0^T \!\int_{\Omega} -\boldsymbol{\xi} \cdot \partial_{t} \nabla \psi_{\varepsilon}\,\mathrm{d} x \, \mathrm{d} t= \int_0^T \!\int_{\Omega} \operatorname{div} \boldsymbol{\xi} \,\partial_{t} \psi_{\varepsilon}\,\mathrm{d} x \, \mathrm{d} t\\
=& \int_0^T \!\int_{\Omega} \operatorname{div} \boldsymbol{\xi} \,\sqrt{2 \rho_\varepsilon f} \,\partial_{t} c_{\varepsilon}\,\mathrm{d} x \, \mathrm{d} t=\int_0^T \!\int_{\Omega} \operatorname{div} \boldsymbol{\xi} \,\sqrt{2 \rho_\varepsilon f} \,(-\frac{\mu_\varepsilon }{\rho_\varepsilon }-\mathbf{v}_\varepsilon  \cdot \nabla c_\varepsilon)\,\mathrm{d} x \, \mathrm{d} t\\
=&\int_0^T \!\int_{\Omega} -\frac{1}{2}(\operatorname{div} \boldsymbol{\xi} \,\sqrt{\frac{2 f}{ \rho_\varepsilon }} + \mu_\varepsilon)^2- \operatorname{div} \boldsymbol{\xi}\mathbf{v}_\varepsilon \mathbf{n}_\varepsilon |\nabla \psi_\varepsilon|+\frac{1}{2}(\operatorname{div} \boldsymbol{\xi} \,\sqrt{\frac{2 f}{\rho_\varepsilon }})^2 +\frac{1}{2}(\mu_\varepsilon)^2\,\mathrm{d} x \, \mathrm{d} t.
\end{split}
\end{eqnarray}
$\bullet$ On the other hand, the following equality holds:
\begin{small}
\begin{align}
&-\int_0^T \!\int_{\Omega} \nabla \psi_{\varepsilon} \cdot \partial_{t} \boldsymbol{\xi} \,\mathrm{d} x \, \mathrm{d} t \notag\\
=&-\int_0^T \!\int_{\Omega}\left(\partial_{t} \boldsymbol{\xi}+(\mathbf{v} \cdot \nabla) \boldsymbol{\xi}+(\nabla \mathbf{v})^{\top} \boldsymbol{\xi}\right) \cdot \mathbf{n}_{\varepsilon}\left|\nabla \psi_{\varepsilon}\right| \,\mathrm{d} x \, \mathrm{d} t+\int_0^T \!\int_{\Omega}\left((\mathbf{v} \cdot \nabla) \boldsymbol{\xi}+(\nabla \mathbf{v})^{\top} \boldsymbol{\xi} \right) \cdot \mathbf{n}_{\varepsilon}\left|\nabla \psi_{\varepsilon}\right| \,\mathrm{d} x \, \mathrm{d} t \notag\\
=&-\int_0^T \!\int_{\Omega}\left(\partial_{t} \boldsymbol{\xi}+(\mathbf{v} \cdot \nabla) \boldsymbol{\xi}+(\nabla \mathbf{v})^{\top} \boldsymbol{\xi}\right) \cdot (\mathbf{n}_{\varepsilon}-\boldsymbol{\xi}) \left|\nabla \psi_{\varepsilon}\right| \,\mathrm{d} x \, \mathrm{d} t-\int_0^T \!\int_{\Omega} \boldsymbol{\xi} \cdot\left(\partial_{t}+(\mathbf{v} \cdot \nabla)\right) \boldsymbol{\xi}\left|\nabla \psi_{\varepsilon}\right| \,\mathrm{d} x \, \mathrm{d} t \notag\\
&-\int_0^T \!\int_{\Omega} \nabla \mathbf{v}: \boldsymbol{\xi} \otimes (\boldsymbol{\xi}-\mathbf{n}_{\varepsilon})\left|\nabla \psi_{\varepsilon}\right| \,\mathrm{d} x \, \mathrm{d} t+\int_0^T \!\int_{\Omega} (\mathbf{v} \cdot \nabla) \boldsymbol{\xi} \, \nabla \psi_{\varepsilon} \,\mathrm{d} x \, \mathrm{d} t \label{2.19.3}.
\end{align}
\end{small}
First, the last term in the right-hand side in (\ref{2.19.3}) can be dealt with by using the fact that $\operatorname{div}\mathbf{v}=0$.
\begin{align*}
\int_0^T \!\int_{\Omega} (\mathbf{v} \cdot \nabla) \boldsymbol{\xi} \, \nabla \psi_{\varepsilon} \,\mathrm{d} x \, \mathrm{d} t =&-\int_0^T \!\int_{\Omega} \mathbf{v} \otimes \boldsymbol{\xi} : \nabla^2 \psi_{\varepsilon} \,\mathrm{d} x \, \mathrm{d} t-\int_0^T \!\int_{\Omega} \operatorname{div}\mathbf{v}\, \boldsymbol{\xi}\cdot \nabla \psi_{\varepsilon} \,\mathrm{d} x \, \mathrm{d} t\\
=&\int_0^T \!\int_{\Omega} (\boldsymbol{\xi}\cdot \nabla) \mathbf{v} \cdot \nabla \psi_{\varepsilon} \,\mathrm{d} x \, \mathrm{d} t+\int_0^T \!\int_{\Omega} \operatorname{div}\boldsymbol{\xi} \,\mathbf{v} \cdot \nabla \psi_{\varepsilon} \,\mathrm{d} x \, \mathrm{d} t\\
=&\int_0^T \!\int_{\Omega} \nabla \mathbf{v}: \mathbf{n}_{\varepsilon} \otimes \boldsymbol{\xi} \,|\nabla \psi_{\varepsilon}| \,\mathrm{d} x \, \mathrm{d} t+\int_0^T \!\int_{\Omega} \operatorname{div}\boldsymbol{\xi} \,\mathbf{v} \cdot \nabla \psi_{\varepsilon} \,\mathrm{d} x \, \mathrm{d} t.
\end{align*}
Second, employing $\operatorname{div}\mathbf{v}=0$ again, we rearrange the last two terms in the right-hand side in (\ref{2.19.3}) as follows:
\begin{eqnarray}
\begin{split} \label{2.19.4}
&-\int_0^T \!\int_{\Omega} \nabla \mathbf{v}: \boldsymbol{\xi} \otimes (\boldsymbol{\xi}-\mathbf{n}_{\varepsilon})\left|\nabla \psi_{\varepsilon}\right| \,\mathrm{d} x \, \mathrm{d} t+\int_0^T \!\int_{\Omega} (\mathbf{v} \cdot \nabla) \boldsymbol{\xi} \, \nabla \psi_{\varepsilon} \,\mathrm{d} x \, \mathrm{d} t\\
=&-\int_0^T \!\int_{\Omega} \nabla \mathbf{v}: \boldsymbol{\xi} \otimes (\boldsymbol{\xi}-\mathbf{n}_{\varepsilon})\left|\nabla \psi_{\varepsilon}\right| \,\mathrm{d} x \, \mathrm{d} t\\
&+\int_0^T \!\int_{\Omega} \nabla \mathbf{v}: \mathbf{n}_{\varepsilon} \otimes \boldsymbol{\xi} \,|\nabla \psi_{\varepsilon}| \,\mathrm{d} x \, \mathrm{d} t+\int_0^T \!\int_{\Omega} \operatorname{div}\boldsymbol{\xi} \,\mathbf{v} \cdot \nabla \psi_{\varepsilon} \,\mathrm{d} x \, \mathrm{d} t\\
=&-\int_0^T \!\int_{\Omega} \nabla \mathbf{v}: (\boldsymbol{\xi}-\mathbf{n}_{\varepsilon}) \otimes (\boldsymbol{\xi}-\mathbf{n}_{\varepsilon})\left|\nabla \psi_{\varepsilon}\right| \,\mathrm{d} x \, \mathrm{d} t\\
&+\int_0^T \!\int_{\Omega} \nabla \mathbf{v}: \mathbf{n}_{\varepsilon} \otimes \mathbf{n}_{\varepsilon} \,|\nabla \psi_{\varepsilon}| \,\mathrm{d} x \, \mathrm{d} t+\int_0^T \!\int_{\Omega} \operatorname{div}\boldsymbol{\xi} \,\mathbf{v} \cdot \nabla \psi_{\varepsilon} \,\mathrm{d} x \, \mathrm{d} t\\
=&-\int_0^T \!\int_{\Omega} \nabla \mathbf{v}: (\boldsymbol{\xi}-\mathbf{n}_{\varepsilon}) \otimes (\boldsymbol{\xi}-\mathbf{n}_{\varepsilon})\left|\nabla \psi_{\varepsilon}\right| \,\mathrm{d} x \, \mathrm{d} t\\
&-\int_0^T \!\int_{\Omega} \nabla \mathbf{v}: (I-\mathbf{n}_{\varepsilon} \otimes \mathbf{n}_{\varepsilon}) |\nabla \psi_{\varepsilon}| \,\mathrm{d} x \, \mathrm{d} t+\int_0^T \!\int_{\Omega} \operatorname{div}\boldsymbol{\xi} \,\mathbf{v} \cdot \nabla \psi_{\varepsilon} \,\mathrm{d} x \, \mathrm{d} t.
\end{split}
\end{eqnarray}
Therefore, combining \eqref{2.19.1}-\eqref{2.19.4} together gives
\begin{align*}
&\quad E \left[\rho_{\varepsilon},\mathbf{v}_{\varepsilon}, c_{\varepsilon} \mid \rho, \mathbf{v}, \chi\right](T)- E \left[\rho_{\varepsilon},\mathbf{v}_{\varepsilon}, c_{\varepsilon} \mid \rho, \mathbf{v}, \chi\right](0)\\
&+\int_0^T \!\int_{\Omega}  (\nabla \mathbf{w})^2 +\frac{1}{2}(\mu_\varepsilon )^{2}\,\mathrm{d} x \, \mathrm{d} t+\int_0^T \!\int_{\Omega}  \frac{1}{2}(\operatorname{div} \boldsymbol{\xi} \,\sqrt{\frac{2 f}{ \rho_\varepsilon }} + \mu_\varepsilon)^2\,\mathrm{d} x \, \mathrm{d} t\\
=&\int_0^T \!\int_{\Omega}  -\mathbf{w}[(\rho_\varepsilon-\rho)(\partial_{t} \mathbf{v}+\mathbf{v}\nabla\mathbf{v}+\nabla \rho )+\rho_\varepsilon \mathbf{w} \nabla \mathbf{v}]\,\mathrm{d} x \, \mathrm{d} t\\
&+\int_0^T \!\int_{\Omega}  \nabla \mathbf{v}:(I-\mathbf{n}_{\varepsilon} \otimes \mathbf{n}_{\varepsilon})(\varepsilon |\nabla c_{\varepsilon}|^2-|\nabla \psi_{\varepsilon}|) \,\mathrm{d} x \, \mathrm{d} t\\
&+\int_0^T \!\int_{\Omega}  -c_0 \chi(\mathbf{w}\cdot \nabla)(\nabla \cdot\boldsymbol{\xi})-\operatorname{div} \boldsymbol{\xi}\,\mathbf{w} \,\mathbf{n}_\varepsilon |\nabla \psi_\varepsilon|\,\mathrm{d} x \, \mathrm{d} t+\int_0^T \!\int_{\Omega}  \frac{1}{2}(\operatorname{div} \boldsymbol{\xi} \,\sqrt{\frac{2 f}{\rho_\varepsilon }})^2 \,\mathrm{d} x \, \mathrm{d} t\\
&-\int_0^T \!\int_{\Omega} \left(\partial_{t} \boldsymbol{\xi}+(\mathbf{v} \cdot \nabla) \boldsymbol{\xi}+(\nabla \mathbf{v})^{\top} \boldsymbol{\xi}\right) \cdot (\mathbf{n}_{\varepsilon}-\boldsymbol{\xi}) \left|\nabla \psi_{\varepsilon}\right| \,\mathrm{d} x \, \mathrm{d} t\\
&-\int_0^T \!\int_{\Omega}  \boldsymbol{\xi} \cdot\left(\partial_{t}+(\mathbf{v} \cdot \nabla)\right) \boldsymbol{\xi}\left|\nabla \psi_{\varepsilon}\right| \,\mathrm{d} x \, \mathrm{d} t\\
&-\int_0^T \!\int_{\Omega}  \nabla \mathbf{v}: (\boldsymbol{\xi}-\mathbf{n}_{\varepsilon}) \otimes (\boldsymbol{\xi}-\mathbf{n}_{\varepsilon})\left|\nabla \psi_{\varepsilon}\right| \,\mathrm{d} x \, \mathrm{d} t \triangleq \; \sum\limits_{i=1}^7 I_i.
\end{align*}
$\bullet$ Next, we will estimate each term $I_i\ (i=1,...,7)$ below.

It follows from Definition \ref{def2.1} and H\"{o}lder's inequality that
\begin{eqnarray}
\begin{split} \label{2.21.2}
|I_1|=&\left|\int_0^T \!\int_{\Omega}  -\mathbf{w}[(\rho_\varepsilon-\rho)(\partial_{t} \mathbf{v}+\mathbf{v}\nabla\mathbf{v}+\nabla \rho )+\rho_\varepsilon \mathbf{w} \nabla \mathbf{v}]\,\mathrm{d} x \, \mathrm{d} t \right| \\
\le & C \int_0^T \!\int_{\Omega}  (\rho_\varepsilon \mathbf{w}^2)\,\mathrm{d} x \, \mathrm{d} t +C \int_0^T \!\int_{\Omega}  (\rho_\varepsilon-\rho)^2 \,\mathrm{d} x \, \mathrm{d} t  \le C \int_0^T \! E \left[\rho_{\varepsilon},\mathbf{v}_{\varepsilon}, c_{\varepsilon} \mid \rho, \mathbf{v}, \chi\right](t) \, \mathrm{d} t .
\end{split}
\end{eqnarray}
To estimate $I_2$, Lemma \ref{lemma 2.3} implies that
\begin{equation}
I_2=\int_0^T \!\int_{\Omega}  \nabla \mathbf{v}: (I-\mathbf{n}_{\varepsilon} \otimes \mathbf{n}_{\varepsilon}) (\varepsilon|\nabla c_{\varepsilon}|^2- |\nabla \psi_{\varepsilon}|) \,\mathrm{d} x \, \mathrm{d} t  \leq C \int_0^T \! E \left[\rho_{\varepsilon},\mathbf{v}_{\varepsilon}, c_{\varepsilon} \mid \rho, \mathbf{v}, \chi\right](t) \, \mathrm{d} t .
\end{equation}
As a consequence of Lemma \ref{lemma2.2} and (\ref{2.28.1}),  we obtain $I_3$ satisfies that
\begin{eqnarray}
\begin{split}
|I_3|=&\left| \int_0^T \!\int_{\Omega}  -c_0 \chi(\mathbf{w}\cdot \nabla)(\nabla \cdot\boldsymbol{\xi})-\operatorname{div} \boldsymbol{\xi}\,\mathbf{w} \,\mathbf{n}_\varepsilon |\nabla \psi_\varepsilon|\,\mathrm{d} x \, \mathrm{d} t \right| \\
=&\left| \int_0^T \!\int_{\Omega}  -(c_0 \chi-\psi_\varepsilon)(\mathbf{w}\cdot \nabla)(\nabla \cdot\boldsymbol{\xi})\,\mathrm{d} x \, \mathrm{d} t  \right| \\
\le& \frac{C}{\eta} \int_0^T \! (E \left[\rho_{\varepsilon},\mathbf{v}_{\varepsilon}, c_{\varepsilon} \mid \rho, \mathbf{v}, \chi\right](t) +E_{\mathrm{vol}}\left[c_{\varepsilon} \mid \chi\right](t) )\, \mathrm{d} t +\eta \int_0^T \!\int_{\Omega}  |\nabla \mathbf{w}|^2 \,\mathrm{d} x \, \mathrm{d} t .
\end{split}
\end{eqnarray}
For $I_4$, we take the following approach. For any $f \in L^{\infty}(0, 2 \delta)$,
\begin{align} \label{9.14.1}
\left(\int_0^{2 \delta}|f(r)| \mathrm{d} r\right)^2 \leq 2\|f\|_{L^{\infty}(0, 2 \delta)} \int_0^{2 \delta}|f(r)| r \mathrm{~d} r,
\end{align}
which is a corollary of Fubini's theorem. Then
\begin{align} \label{3.5.2}
&I_4=\int_0^T \!\int_{\Omega} \frac{1}{2}(\operatorname{div} \boldsymbol{\xi} \,\sqrt{\frac{2 f}{\rho_\varepsilon }})^2 \, \mathrm{d} x \, \mathrm{d} t =
\int_0^T \!\int_{\Gamma_t} \int_{-2 \delta}^{2 \delta}\frac{1}{2}(\frac{\operatorname{div}\boldsymbol{\xi}}{\rho_\varepsilon })^2  \,2  \rho_\varepsilon f \,J(r,p,t)\, \mathrm{d} r\,\mathrm{d} \sigma(p) \, \mathrm{d} t \notag\\
\le& C \int_0^T \!\int_{\Gamma_t}\sqrt{ \| 2 \rho_\varepsilon f\|_{L^\infty(-2\delta,2\delta)} \int_{-2 \delta}^{2 \delta} \,2 \rho_\varepsilon f \cdot r\, \mathrm{d} r}\,\mathrm{d} \sigma(p) \, \mathrm{d} t \notag\\
\le& C \int_0^T \!\int_{\Gamma_t} \varepsilon^\frac{1}{4} \left(\int_{-2 \delta}^{2 \delta} \sqrt{\frac{2 \rho_\varepsilon f}{\varepsilon}} \cdot r\, \mathrm{d} r \right)^\frac{1}{2}\,\mathrm{d} \sigma(p) \, \mathrm{d} t \le C \varepsilon^\frac{1}{3}+ C \int_0^T \!\int_{\Gamma_t} \left(\int_{-2 \delta}^{2 \delta} \sqrt{\frac{2 \rho_\varepsilon f}{\varepsilon}} \cdot r\, \mathrm{d} r \right)^2\,\mathrm{d} \sigma(p) \, \mathrm{d} t \notag\\
\le& C \varepsilon^\frac{1}{3}+ C \int_0^T \!\int_{\Gamma_t} \int_{-2 \delta}^{2 \delta} \frac{2 \rho_\varepsilon f}{\varepsilon} \cdot r^2\, \mathrm{d} r \,\mathrm{d} \sigma(p) \, \mathrm{d} t \le C \varepsilon^\frac{1}{3}+ C \int_0^T \! E \left[\rho_{\varepsilon},\mathbf{v}_{\varepsilon}, c_{\varepsilon} \mid \rho, \mathbf{v}, \chi\right](t) \, \mathrm{d} t ,
\end{align}
where Proposition \ref{prop2.1}, \eqref {9.14.1} and the fact that $\| 2 \rho_\varepsilon f\|_{L^\infty(-2\delta,2\delta)} \le C$ are used.

Applying property (A4) of $\boldsymbol{\xi}$ and Proposition \ref{prop2.1} again, we have
\begin{align}
|I_5|=&\left|-\int_0^T \!\int_{\Omega} \left(\partial_{t} \boldsymbol{\xi}+(\mathbf{v} \cdot \nabla) \boldsymbol{\xi}+(\nabla \mathbf{v})^{\top} \boldsymbol{\xi}\right) \cdot (\mathbf{n}_{\varepsilon}-\boldsymbol{\xi}) \left|\nabla \psi_{\varepsilon}\right|\,\mathrm{d} x \, \mathrm{d} t  \right| \notag\\
\le& C \int_0^T \! E \left[\rho_{\varepsilon},\mathbf{v}_{\varepsilon}, c_{\varepsilon} \mid \rho, \mathbf{v}, \chi\right](t) \, \mathrm{d} t ,
\end{align}
\begin{equation}
|I_6|=\left|-\int_0^T \!\int_{\Omega}  \boldsymbol{\xi} \cdot\left(\partial_{t}+(\mathbf{v} \cdot \nabla)\right) \boldsymbol{\xi}\left|\nabla \psi_{\varepsilon}\right|\,\mathrm{d} x \, \mathrm{d} t \right| \le C \int_0^T \! E \left[\rho_{\varepsilon},\mathbf{v}_{\varepsilon}, c_{\varepsilon} \mid \rho, \mathbf{v}, \chi\right](t) \, \mathrm{d} t
\end{equation}
and
\begin{equation} \label{2.21.3}
|I_7|=\left|-\int_0^T \!\int_{\Omega}  \nabla \mathbf{v}: (\boldsymbol{\xi}-\mathbf{n}_{\varepsilon}) \otimes (\boldsymbol{\xi}-\mathbf{n}_{\varepsilon})\left|\nabla \psi_{\varepsilon}\right|\,\mathrm{d} x \, \mathrm{d} t \right| \le C \int_0^T \! E \left[\rho_{\varepsilon},\mathbf{v}_{\varepsilon}, c_{\varepsilon} \mid \rho, \mathbf{v}, \chi\right](t) \, \mathrm{d} t .
\end{equation}
Using (\ref{2.21.2})-(\ref{2.21.3}) and choosing $\eta$ suitably small, we arrive at
\begin{align*}
& E \left[\rho_{\varepsilon},\mathbf{v}_{\varepsilon}, c_{\varepsilon} \mid \rho, \mathbf{v}, \chi\right](T)+\int_0^T \!\int_{\Omega} \frac{1}{2}(\nabla \mathbf{w})^2 + \frac{1}{2}(\mu_\varepsilon )^{2}\,\mathrm{d} x \, \mathrm{d} t+\int_0^T \!\int_{\Omega} \frac{1}{2}(\operatorname{div} \boldsymbol{\xi} \,\sqrt{\frac{2 f}{ \rho_\varepsilon }} + \mu_\varepsilon)^2 \,\mathrm{d} x \, \mathrm{d} t\\
\le& E \left[\rho_{\varepsilon},\mathbf{v}_{\varepsilon}, c_{\varepsilon} \mid \rho, \mathbf{v}, \chi\right](0)+ C \varepsilon^\frac{1}{3}+ C \int_0^T (E \left[\rho_{\varepsilon},\mathbf{v}_{\varepsilon}, c_{\varepsilon} \mid \rho, \mathbf{v}, \chi\right](t)+E_{\mathrm{vol}}\left[c_{\varepsilon} \mid \chi\right](t))\, \mathrm{d} t.
\end{align*}
Then the proof of Proposition \ref{prop3.2} is done.
\end{proof}

\subsection{Interface Error Estimate}
We now turn to derive the estimate for the bulk error.
\begin{prop}\label{pro3.2}
There exists a generic constant $C$, such that for any $T \in (0,T_0)$, $E_{\mathrm{vol}}\left[c_{\varepsilon} \mid \chi\right]$ in (\ref{9.13.1}) satisfies the following estimate:
\begin{eqnarray}
\begin{split} \label{2.21.5}
& E_{\mathrm{vol}}\left[c_{\varepsilon} \mid \chi\right](T)+\int_0^T \!\int_{\Omega}  \varepsilon (\frac{\mu_\varepsilon }{2} +\frac{\vartheta}{\rho_{\varepsilon}} |\nabla c_{\varepsilon}|)^2 \,\mathrm{d} x \, \mathrm{d} t \\
\le & E_{\mathrm{vol}}\left[c_{\varepsilon} \mid \chi\right](0)+ C(\eta)\int_0^T (E \left[\rho_{\varepsilon},\mathbf{v}_{\varepsilon}, c_{\varepsilon} \mid \rho, \mathbf{v}, \chi\right](t)+E_{\mathrm{vol}}\left[c_{\varepsilon} \mid \chi\right](t))\, \mathrm{d} t \\
&+\eta \int_0^T \!\int_{\Omega} |\nabla \mathbf{w}|^2 \,\mathrm{d} x \, \mathrm{d} t+(\frac{1}{4}+\eta) \varepsilon \int_0^T \!\int_{\Omega} (\mu_\varepsilon)^2 \,\mathrm{d} x \, \mathrm{d} t
\end{split}
\end{eqnarray}
holds for a suitably small constant $\eta$.
\end{prop}

\begin{proof}
Note  that
$$
c_0 \int_0^T \!\int_{\Omega} \vartheta \partial_t \chi \,\mathrm{d} x \, \mathrm{d} t=0,
$$
and according to the condition that $\vartheta =0$ along $\Gamma_t$, after integration by parts, we obtain that
\begin{align}
\label{EBE}
&E_{\mathrm{vol}}\left[c_{\varepsilon} \mid \chi\right](T)-E_{\mathrm{vol}}\left[c_{\varepsilon} \mid \chi\right](0) \notag\\
=& \int_0^T \!\int_{\Omega} \left(\psi_{\varepsilon}-c_0\chi\right) \partial_t \vartheta \,\mathrm{d} x \, \mathrm{d} t +\int_0^T \!\int_{\Omega}  \vartheta \partial_t \psi_{\varepsilon} \,\mathrm{d} x \, \mathrm{d} t\nonumber\\
=&\int_0^T \!\int_{\Omega} \left(\psi_{\varepsilon}-c_0\chi\right)\left(\partial_t \vartheta+(\mathbf{v}\cdot \nabla) \vartheta\right)\,\mathrm{d} x \, \mathrm{d} t-\int_0^T \!\int_{\Omega} \left(\psi_{\varepsilon}-c_0\chi\right)\nabla \cdot (\mathbf{v} \vartheta) \,\mathrm{d} x \, \mathrm{d} t +\int_0^T \!\int_{\Omega}  \vartheta \partial_t \psi_{\varepsilon} \,\mathrm{d} x \, \mathrm{d} t\nonumber\\
=& \int_0^T \!\int_{\Omega} \left(\psi_{\varepsilon}-c_0 \chi\right) \left(\partial_t \vartheta+(\mathbf{v} \cdot \nabla) \vartheta\right) \,\mathrm{d} x \, \mathrm{d} t+ \int_0^T \!\int_{\Omega}  \vartheta (\partial_t \psi_{\varepsilon}+(\mathbf{v}\cdot \nabla) \psi_{\varepsilon}) \,\mathrm{d} x \, \mathrm{d} t
\end{align}
for a.e. $T \in (0, T_0)$.\\
For the last term in (\ref{EBE}), one has
\begin{align*}
& \int_0^T \!\int_{\Omega}  \vartheta\left(\partial_t \psi_{\varepsilon}+(\mathbf{v} \cdot \nabla) \psi_{\varepsilon}\right)\,\mathrm{d} x \, \mathrm{d} t \\
= & \int_0^T \!\int_{\Omega}  \vartheta(\mathbf{v} \cdot \nabla) \psi_{\varepsilon} \,\mathrm{d} x \, \mathrm{d} t -\int_0^T \!\int_{\Omega}  \vartheta\left(\sqrt{\varepsilon }\frac{\mu_\varepsilon }{\rho_\varepsilon } \frac{\sqrt{2 \rho_{\varepsilon} f}}{\sqrt{\varepsilon}}+\left(\mathbf{v}_{\varepsilon} \cdot \nabla\right) \psi_{\varepsilon}\right) \,\mathrm{d} x \, \mathrm{d} t\\
= &\int_0^T \!\int_{\Omega}  -\vartheta \sqrt{\varepsilon }\frac{\mu_\varepsilon }{\rho_\varepsilon } \frac{\sqrt{2 \rho_{\varepsilon} f}}{\sqrt{\varepsilon}} \,\mathrm{d} x \, \mathrm{d} t+\int_0^T \!\int_{\Omega}  \vartheta\left(\left(\mathbf{v}-\mathbf{v}_{\varepsilon}\right) \cdot \nabla\right) \psi_{\varepsilon} \,\mathrm{d} x \, \mathrm{d} t
\triangleq \; \sum\limits_{i=1}^2 J_i.
\end{align*}
Now, we start with the estimate of $J_1$. Adding zero and then using (\ref{9.17.3}) to obtain
\begin{small}
\begin{eqnarray}
\begin{split} \label{2.20.2}
&J_1=\int_0^T \!\int_{\Omega}  -\vartheta \sqrt{\varepsilon }\frac{\mu_\varepsilon }{\rho_\varepsilon } \frac{\sqrt{2 \rho_{\varepsilon} f}}{\sqrt{\varepsilon}} \,\mathrm{d} x \, \mathrm{d} t\\
=&\int_0^T \!\int_{\Omega} - \vartheta \sqrt{\varepsilon }\frac{\mu_\varepsilon }{\rho_\varepsilon } (\frac{\sqrt{2 \rho_{\varepsilon} f}}{\sqrt{\varepsilon}}-\sqrt{\varepsilon}|\nabla c_{\varepsilon}|) +\frac{1}{4}( \sqrt{\varepsilon} \mu_\varepsilon)^2 +(\frac{\vartheta}{\rho_{\varepsilon}} \sqrt{\varepsilon}|\nabla c_{\varepsilon}|)^2\,\mathrm{d} x \, \mathrm{d} t-\int_0^T \!\int_{\Omega}  \varepsilon(\frac{\mu_\varepsilon }{2} +\frac{\vartheta}{\rho_{\varepsilon}} |\nabla c_{\varepsilon}|)^2\,\mathrm{d} x \, \mathrm{d} t\\
\le& \int_0^T \!\int_{\Omega}  C(\frac{\sqrt{2 \rho_{\varepsilon} f}}{\sqrt{\varepsilon}}-\sqrt{\varepsilon}|\nabla c_{\varepsilon}|)^2 +(\frac{1}{4}+\eta) {\varepsilon} (\mu_\varepsilon)^2 + \varepsilon (\frac{\vartheta}{\rho_{\varepsilon}})^2( |\nabla c_{\varepsilon}|)^2\,\mathrm{d} x \, \mathrm{d} t-\int_0^T \!\int_{\Omega}  \varepsilon (\frac{\mu_\varepsilon }{2} +\frac{\vartheta}{\rho_{\varepsilon}} |\nabla c_{\varepsilon}|)^2\,\mathrm{d} x \, \mathrm{d} t\\
\le& C \int_0^T \! E \left[\rho_{\varepsilon},\mathbf{v}_{\varepsilon}, c_{\varepsilon} \mid \rho, \mathbf{v}, \chi\right](t) \, \mathrm{d} t + \int_0^T \!\int_{\Omega}  (\frac{1}{4}+\eta){\varepsilon} (\mu_\varepsilon)^2 \,\mathrm{d} x \, \mathrm{d} t -\int_0^T \!\int_{\Omega}  \varepsilon (\frac{\mu_\varepsilon }{2} +\frac{\vartheta}{\rho_{\varepsilon}} |\nabla c_{\varepsilon}|)^2 \,\mathrm{d} x \, \mathrm{d} t,
\end{split}
\end{eqnarray}
\end{small}
where $\eta$ is a small constant and $C$ depending on $ \vartheta$ and $\rho_\varepsilon $. Moreover, (\ref{9.17.5}), as well as Proposition \ref{prop2.1}, are used in the last inequality.

For $J_2$, (\ref{9.17.4}) and (\ref{2.19.7}) give that
\begin{eqnarray}
\begin{split} \label{2.20.3}
J_2=& \int_0^T \!\int_{\Omega} \vartheta\left(\left(\mathbf{v}-\mathbf{v}_{\varepsilon}\right) \cdot \nabla\right) \psi_{\varepsilon} \,\mathrm{d} x \, \mathrm{d} t\\
=& -\int_0^T \!\int_{\Omega} \psi_{\varepsilon}\left(\left(\mathbf{v}-\mathbf{v}_{\varepsilon}\right) \cdot \nabla\right) \vartheta \,\mathrm{d} x \, \mathrm{d} t= \int_0^T \!\int_{\Omega} \left(c_0 \chi-\psi_{\varepsilon}\right)\left(\left(\mathbf{v}-\mathbf{v}_{\varepsilon}\right) \cdot \nabla\right) \vartheta \,\mathrm{d} x \, \mathrm{d} t \\
\le& \frac{C}{\eta} \int_0^T (E \left[\rho_{\varepsilon},\mathbf{v}_{\varepsilon}, c_{\varepsilon} \mid \rho, \mathbf{v}, \chi\right](t)+E_{\mathrm{vol}}\left[c_{\varepsilon} \mid \chi\right](t))\, \mathrm{d} t  +\eta \int_0^T \!\int_{\Omega} |\nabla \mathbf{v}_{\varepsilon}-\nabla \mathbf{v}|^2 \,\mathrm{d} x \, \mathrm{d} t.
\end{split}
\end{eqnarray}
Consequently, (\ref{9.17.6}), (\ref{EBE}), (\ref{2.20.2}) and (\ref{2.20.3}) yield
\begin{align*}
& E_{\mathrm{vol}}\left[c_{\varepsilon} \mid \chi\right](T)+\int_0^T \!\int_{\Omega}  \varepsilon (\frac{\mu_\varepsilon }{2} +\frac{\vartheta}{\rho_{\varepsilon}} |\nabla c_{\varepsilon}|)^2 \,\mathrm{d} x \, \mathrm{d} t \\
\le & E_{\mathrm{vol}}\left[c_{\varepsilon} \mid \chi\right](0)+ \frac{C}{\eta} \int_0^T (E \left[\rho_{\varepsilon},\mathbf{v}_{\varepsilon}, c_{\varepsilon} \mid \rho, \mathbf{v}, \chi\right](t)+E_{\mathrm{vol}}\left[c_{\varepsilon} \mid \chi\right](t))\, \mathrm{d} t \\
&+\eta \int_0^T \!\int_{\Omega} |\nabla \mathbf{w}|^2 \,\mathrm{d} x \, \mathrm{d} t+(\frac{1}{4}+\eta) \varepsilon \int_0^T \!\int_{\Omega} (\mu_\varepsilon)^2 \,\mathrm{d} x \, \mathrm{d} t.
\end{align*}
The the proof of Proposition \ref{pro3.2} is complete.
\end{proof}

\section{Proofs of Theorem \ref{thm1.1} and Corollary \ref{cor1.2}}

With the help of the a prior estimates established in previous section, we are ready to prove Theorem \ref{thm1.1}.
\begin{proof}
To prove Theorem \ref{thm1.1}, we use (\ref{9.12.3}) and  (\ref{2.21.5}) to obtain
\begin{eqnarray} \label{3.5.1}
\begin{split}
&\quad E \left[\rho_{\varepsilon},\mathbf{v}_{\varepsilon}, c_{\varepsilon} \mid \rho, \mathbf{v}, \chi\right](t)+E_{\mathrm{vol}}\left[c_{\varepsilon} \mid \chi\right](t)+[\frac{1}{2}-(\frac{1}{4}+\eta)\varepsilon]\int_0^t \!\int_{\Omega} (\mu_\varepsilon )^{2}\,\mathrm{d}x\,\mathrm{d}\varsigma\\
&+\frac{1}{3} \int_0^t \!\int_{\Omega} (\nabla \mathbf{w})^2\,\mathrm{d}x\,\mathrm{d}\varsigma +\int_0^t \!\int_{\Omega} \frac{1}{2}(\operatorname{div} \boldsymbol{\xi} \,\sqrt{\frac{2 f}{ \rho_\varepsilon }} + \mu_\varepsilon)^2 \,\mathrm{d}x\,\mathrm{d}\varsigma+\int_0^t \!\int_{\Omega}  \varepsilon (\frac{\mu_\varepsilon }{2} +\frac{\vartheta}{\rho_{\varepsilon}} |\nabla c_{\varepsilon}|)^2 \,\mathrm{d}x\,\mathrm{d}\varsigma \\
\le& E \left[\rho_{\varepsilon},\mathbf{v}_{\varepsilon}, c_{\varepsilon} \mid \rho, \mathbf{v}, \chi\right](0)+E_{\mathrm{vol}}\left[c_{\varepsilon} \mid \chi\right](0)\\
&+C \varepsilon^\frac{1}{3}+ C \int_0^t (E \left[\rho_{\varepsilon},\mathbf{v}_{\varepsilon}, c_{\varepsilon} \mid \rho, \mathbf{v}, \chi\right](\varsigma)+E_{\mathrm{vol}}\left[c_{\varepsilon} \mid \chi\right](\varsigma))\, \mathrm{d}\varsigma
\end{split}
\end{eqnarray}
for all $t \in (0,T)$, provided that $\varepsilon$ and $\eta$ are chosen to be suitably small.

Thus, by Gronwall's inequality, we conclude that
\begin{eqnarray*}
\begin{split}
&\quad E \left[\rho_{\varepsilon},\mathbf{v}_{\varepsilon}, c_{\varepsilon} \mid \rho, \mathbf{v}, \chi\right](t)+ E_{\mathrm{vol}}\left[c_{\varepsilon} \mid \chi\right](t) + \frac{1}{3} \| (\nabla \mathbf{w})\|_{L^2((0,t)\times \Omega)}^2+\frac{1}{3} \| \mu_\varepsilon\|_{L^2((0,t)\times \Omega)}^2\\
&+\int_0^t \int_{\Omega} \frac{1}{2}(\operatorname{div} \boldsymbol{\xi} \,\sqrt{\frac{2 f}{\rho_\varepsilon }} + \mu_\varepsilon )^2 \, \mathrm{d} x \, \mathrm{d} \varsigma +\int_0^t \int_{\Omega} \varepsilon (\frac{\mu_\varepsilon }{2} +\frac{\vartheta}{\rho_{\varepsilon}} |\nabla c_{\varepsilon}|)^2 \,\mathrm{d} x \, \mathrm{d} \varsigma\\
\le& e^{Ct}\left(E \left[\rho_{\varepsilon},\mathbf{v}_{\varepsilon}, c_{\varepsilon} \mid \rho, \mathbf{v}, \chi\right](0)+E_{\mathrm{vol}}\left[c_{\varepsilon} \mid \chi\right](0)+ C t \varepsilon^\frac{1}{3}\right) \le C \varepsilon^\frac{1}{3},
\end{split}
\end{eqnarray*}
provided that the initial condition (\ref{9.15.2}) is satisfied. Consequently, the proof of Theorem \ref{thm1.1} is complete.
\end{proof}

\begin{rem}
\label{RE}
For the general case $m_{\varepsilon}=m_0\,{\varepsilon}^{\theta}$, the results of convergence in Theorem \ref{thm1.1} still hold as long as $-\frac{1}{4} < \theta \le 1$, meanwhile the decay rate will decrease. To this end,
without loss of generality, we take $m_0=1$. It is direct to check that $[\frac{1}{2}-(\frac{1}{4}+\eta)\varepsilon] \int_0^t \!\int_{\Omega} (\mu_\varepsilon )^{2}\,\mathrm{d}x\,\mathrm{d}\varsigma$ in the left-hand side of (\ref{3.5.1}) will be replaced by $[\frac{1}{2} {\varepsilon}^{\theta}-(\frac{1}{4}+\eta)\varepsilon^{2\theta+1}] \int_0^t \!\int_{\Omega} (\mu_\varepsilon )^{2}\,\mathrm{d}x\,\mathrm{d}\varsigma$ while the $\varepsilon^\frac{1}{3}$ on the right-hand side will be replaced by $\varepsilon^{\frac{1}{3}+\frac{4 \theta}{3}}$,  so we only require that $\frac{1}{2} {\varepsilon}^{\theta}-(\frac{1}{4}+\eta)\varepsilon^{2\theta+1} \ge 0$ and $\frac{1}{3}+\frac{4 \theta}{3}>0$. For this reason, it is necessary to impose that $\theta > -\frac{1}{4}$. According the results in \cite{Abels H}, to achieve the sharp interface limit, it is natural to require that  $-\frac{1}{4} < \theta \le 1$.
\end{rem}

Now we are going to prove Corollary \ref{cor1.2}. The proof can be completed by the method which is analogous to that used in the last paragraph in \cite{Fischer J.}.
\begin{proof}
(\ref{9.14.1}) allows us to consider it only in $\Gamma_t(\frac{\delta}{2})$, and we can adopt the same procedure as those used in the proof of \cite[Lemma 5.1]{Liu Y N} to deduce that
\begin{align*}
&\left(\int_{\Gamma_t(\frac{\delta}{2})}|\psi_{\varepsilon}(x, t)-c_0\chi(x, t)| \mathrm{d} x\right)^2 \\
=& \left(\int_{\Gamma_t} \int_{-\frac{\delta}{2}}^{\frac{\delta}{2}}|\psi_{\varepsilon}(p+r \mathbf{n}_{\Gamma_{t}}(p), t)-c_0\chi(p+r \mathbf{n}_{\Gamma_{t}}(p), t)| J(r,p,t)\, \mathrm{d} r\,\mathrm{d} \sigma(p) \right)^2 \\
\leq& C \int_{\Gamma_t} \|\psi_{\varepsilon}(p+r \mathbf{n}_{\Gamma_{t}}(p), t)-c_0\chi(p+r \mathbf{n}_{\Gamma_{t}}(p), t)\|_{L^{\infty}(-\frac{\delta}{2}, \frac{\delta}{2})} \\
&\quad \quad \quad \quad \int_{-\frac{\delta}{2}}^{\frac{\delta}{2}}|\psi_{\varepsilon}(p+r \mathbf{n}_{\Gamma_{t}}(p), t)-c_0\chi(p+r \mathbf{n}_{\Gamma_{t}}(p), t)|\times d_\Gamma (p+r \mathbf{n}_{\Gamma_{t}}(p), \Gamma_t) \mathrm{d} r\,\mathrm{d} \sigma(p) \\
\leq& C \int_{\Gamma_t(\frac{\delta}{2})} |\psi_{\varepsilon}(x, t)-c_0\chi(x, t)| d_\Gamma(x,t) \,\mathrm{d} x,
\end{align*}
where (\ref{9.14.1}) and $\|g\|_{L^1}^2 \le C \|g\|_{L^2}^2$ are used in the first inequality. Moreover, the fact that $\|\psi_{\varepsilon}(p+r \mathbf{n}_{\Gamma_{t}}(p), t)-c_0\chi(p+r \mathbf{n}_{\Gamma_{t}}(p), t)\|_{L^{\infty}(-\frac{\delta}{2}, \frac{\delta}{2})}\le C$ is used in the second inequality.

Finally, employing  (\ref{2.21.7}) and (\ref{9.23.4}) to complete the proof of Corollary \ref{cor1.2}.
\end{proof}
\vspace{.1in}
\noindent{\bf Acknowledgments:} The research of S. Jiang was supported by National Key R\&D Program (2020YFA0712200), National Key Project
(GJXM92579), and NSFC (Grant No. 11631008), the Sino-German Science Center (Grant No. GZ 1465)
and the ISF{NSFC joint research program (Grant No. 11761141008), and the research of F.  Xie was supported by National Natural Science Foundation of China No.12271359, 11831003, 12161141004 and Shanghai Science and Technology Innovation Action Plan No. 20JC1413000.

\end{document}